\documentclass{article}
\usepackage[utf8]{inputenc}

\usepackage[nottoc]{tocbibind}

\usepackage[toc,page]{appendix}

\usepackage{imakeidx}
\makeindex[columns=2, 
           options= -s style.ist]

\usepackage{hyperref}

\usepackage[runin]{abstract}

\usepackage{titlesec}
\titleformat{\section}
  {\normalfont\scshape\large\centering}{\thesection.}{1em}{}

\titleformat{\subsection}
  {\normalfont\scshape\centering}{\thesubsection.}{1em}{}

\usepackage{authblk}

\usepackage{amsthm}
\usepackage{amssymb}
\usepackage{amsmath}
\usepackage{mathtools}
\usepackage[top=3cm, bottom=2cm, left=3cm, right=3cm, a4paper]{geometry}

\usepackage{tikz}
\usepackage{amsfonts}
\usetikzlibrary{positioning}
\usepackage{tikz-cd}

\tikzset{mynode/.style={draw,circle,fill=black,inner sep=1pt,outer sep=0pt}
}

\usepackage{mathrsfs}

\newcommand{\rotvert}{\rotatebox[origin=c]{90}{$\vert$}}

\usepackage{comment}

\usepackage{lscape}

\usepackage{graphicx}

\usepgflibrary{fpu}

\usepackage{enumerate}

\usepackage{blkarray, bigstrut}

\theoremstyle{definition}
\newtheorem{definition}{Definition}[section]
\newtheorem{example}[definition]{Example}

\theoremstyle{plain}

\newtheorem{theorem}[definition]{Theorem}

\newtheorem{proposition}[definition]{Proposition}

\newtheorem{lemma}[definition]{Lemma}
\newtheorem{corollary}[definition]{Corollary}

\theoremstyle{remark}
\newtheorem{remark}[definition]{Remark}

\newtheorem{problem}{\bf{Problem}}

\title{{\bf A Note on Idempotent Matrices: The Poset Structure and The Construction}}
\author{Sen-Peng Eu, Yong-Siang Lin, and Wei-Liang Sun}

\providecommand{\keywords}[1]
{
  \textit{Keywords:} #1
}
\providecommand{\MSC}[1]
{
  \textit{MSC:} #1
}

\newcommand{\Addresses}{{
  \bigskip

  Sen-Peng Eu, \textsc{Department of Mathematics, National Taiwan Normal University, Taipei 11677, Taiwan, ROC}\par\nopagebreak
  \textit{E-mail address}: \texttt{speu@math.ntnu.edu.tw}

  \medskip

  Yong-Siang Lin, \textsc{Institute of Mathematics, Academia Sinica, Taipei 106319, Taiwan, ROC}\par\nopagebreak
  \textit{E-mail address}: \texttt{hsiang0000@gmail.com}

  \medskip

  Wei-Liang Sun, \textsc{Department of Mathematics, National Kaohsiung Normal University, Kaohsiung 824004, Taiwan, ROC}\par\nopagebreak
  \textit{E-mail address}: \texttt{wlsun@mail.nknu.edu.tw}

}}

\date{}

\begin{document}

\maketitle

\begin{abstract}
    Idempotent elements play a fundamental role in ring theory, as they encode significant information about the underlying algebraic structure. In this paper, we study idempotent matrices from two perspectives. First, we analyze the partially ordered set of idempotents in matrix rings over a division ring. We characterize the partial order relation explicitly in terms of block decompositions of idempotent matrices. Second, over principal ideal domains, we establish an equivalent condition for a matrix to be idempotent, derived from matrix factorizations using the Smith normal form. We also consider extensions over unique factorization domains and constructions via the Kronecker product and the anti-transpose. Together, these results clarify both the structural and constructive aspects of idempotents in matrix rings.
    Moreover, the set of idempotent matrices over a field can be viewed as an affine algebraic variety.
\end{abstract}

{\keywords{idempotent matrices, poset, Smith normal form, dimension of an affine algebraic variety.}}

{\MSC{16U40, 15B33.}}

\section{Introduction}

An element $e$ in a ring $R$ is called an idempotent if $e^2=e$. Such elements are crucial in ring theory, as they frequently serve to decompose rings into simpler components. For example, the Wedderburn-Artin theorem asserts that any semisimple Artinian ring with unity decomposes as a finite direct sum of matrix rings over division rings. Classical proofs of this theorem, as \cite[Theorem~IX.3.3]{algebra}, \cite[Theorem~14.15]{mialgebra}, and \cite[Theorem~2.6.18]{groupring}, rely heavily on the Jacobson density theorem.
For some alternative proofs, see \cite{AShortProofofWedderburnArtinTheoremTK} and \cite{ashortproofofwedderburnartintheorem}.

Nicholson \cite{ashortproofofwedderburnartintheorem}, however, showed that the same conclusion can be reached through a shorter and more elegant argument based on structural properties of idempotents. In particular, he introduced a natural partial order on the set of idempotents of a ring. Motivated by this idea, we investigate the poset structure of idempotent matrices over division rings. Our first main result, {Theorem~\ref{thm: all the idempotent greater than E}}, provides a precise characterization of the partial order relation between two idempotents in terms of block matrix decompositions.

The second theme of this paper concerns constructive methods for idempotent matrices over integral domains. When the base ring is a principal ideal domain, we develop a criterion that makes use of the Smith normal form. This leads to our second main result, {Theorem~\ref{thm: idempotent in Mn(PID) in alternative form}}, which gives an equivalent condition for a matrix to be idempotent. We also provide explicit examples over various domains. In addition, we examine further constructions involving the Kronecker product and the anti-transpose.

The paper is organized as follows. {Section~\ref{sec: The Poset Structure of Idempotents}} develops the poset structure of idempotent matrices over division rings, culminating in {Theorem~\ref{thm: all the idempotent greater than E}}. {Section~\ref{sec: Constructing Idempotent Matrices}} is devoted to constructive aspects: {Subsection~\ref{subsec: Idempotents in matrix ring over a UFD}} addresses unique factorization domains, {Subsection~\ref{subsec: Idempotents in matrix ring over a PID}} focuses on principal ideal domains with {Theorem~\ref{thm: idempotent in Mn(PID) in alternative form}} as the main result, while {Subsection~\ref{subsec: Additional Constructions}} presents additional constructions via the Kronecker product and the anti-transpose. {Section~\ref{sec: As an Affine Algebraic Variety}} regards the set of idempotent matrices over a field $K$ as an affine algebraic variety and discusses some of its basic properties.  Finally, we conclude with remarks and possible directions for future research.

\subsection{Notation}
Throughout the paper, $\dot+$ denotes the internal direct sum of submodules, while $\oplus$ denotes the external direct sum. Bold lowercase letters (e.g., ${\bf v},{\bf v}_i$) are used to denote column vectors. In contrast, when a bold symbol carries a subscript on the left (e.g., $_i{\bf v}$), it denotes a row vector.

\section{The Poset Structure of Idempotents}
\label{sec: The Poset Structure of Idempotents}

We first introduce a partial order relation on the set of idempotents in a ring. Let $R$ be a ring and ${\mathscr I}(R)$\index{${\mathscr I}(R)$: the set of idempotents in $R$} be the set of idempotents. That is, 
\[
{\mathscr I}(R)=\{e\in R\mid e^2=e\}.
\]
Because $0_R$ is always an idempotent, the set ${\mathscr I}(R)$ is nonempty.
Define a symbol ``$\le$''\index{$\le$} on the set ${\mathscr I}(R)$ by 
\[
e\le f\ {\textrm{if and only if}}\ ef=e=fe
\]
for $e,f\in {\mathscr I}(R)$.

\begin{proposition}
    \label{prop: idempotent in a ring form a poset}
    Let $R$ be a ring.
    The symbol ``$\le$'' is an partial order relation on ${\mathscr I}(R)$.
    Thus, ${\mathcal P}=({\mathscr I}(R),\le)$ is a poset.
\end{proposition}

This partial order is called the natural partial order on idempotents since $e\le f$ if and only if $eRe\subseteq fRf$ as rings.

\subsection{Preliminary: Linear Algebra over Division Rings}
\label{subsec: Preliminary: Linear Algebra over Division Rings}

Suppose $E$ is an idempotent matrix in $M_n(K)$ where $K$ is a field.
From elementary linear algebra, we know that $E$ is diagonalizable with eigenvalue in $\{0,1\}$. To see this, we note that the minimal polynomial $m_E(x)$ is a divisor of $x^2-x=x(x-1)$.
Thus, we expect that idempotent matrices over a division ring behave as field case.
Let $\Delta$ be a division ring. Although we are unable to define the characteristic polynomial and the minimal polynomial for a matrix in $M_n(\Delta)$, the rank of a matrix is still well-defined as the usual sense that equals to the dimension of the column space. 
For the detail, see \cite[Section~VII.2]{algebra}.
More precisely, for a given $n\in {\mathbb N}$, we have
\begin{enumerate}
    \item $\operatorname{rank}:M_n(\Delta)\to \{0,1,...,n\}$ is a well-defined function.
    \item Let $A,B\in M_n(\Delta)$. Then $A$ is similar to $B$ if and only if $\operatorname{rank}(A)=\operatorname{rank}(B)$.
    \item $\operatorname{rank}(A)=n$ if and only if $A$ is invertible. Therefore, the columns of $A$ form a basis for $\Delta^n$ if and only if $A$ is invertible.
\end{enumerate}

\begin{proposition}
    \label{diagonalization of idempotent matrix}
    Let $E\in {\mathscr I}(M_n(\Delta))$. Then there exists an invertible matrix $A$ and a diagonal matrix 
    \[
    D_E=\left(\begin{array}{c|c}
        I_{\operatorname{rank}(E)} & O \\
        \hline
        O & O
    \end{array}\right)
    \]
    such that $E=AD_EA^{-1}$.
\end{proposition}
\begin{proof}
    We first show that $\Delta^n=\operatorname{Im}(E)\dot+\operatorname{Ker}(E)$ where $\operatorname{Im}(E)=\{E{\bf v}\mid {\bf v}\in \Delta^n\}$ and $\operatorname{Ker}(E)=\{{\bf v}\mid E{\bf v}={\bf 0}\}$.
    Let ${\bf v}\in \Delta^n$. Then $E{\bf v}\in \operatorname{Im}(E)$, ${\bf v}-E{\bf v}\in \operatorname{Ker}(E)$, and ${\bf v}=E{\bf v}+({\bf v}-E{\bf v})\in \operatorname{Im}(E)+\operatorname{Ker}(E)$.
    Let ${\bf u}\in \operatorname{Im}(E)\cap \operatorname{Ker}(E)$. Then ${\bf u}=E{\bf u}'$ for some ${\bf u}'\in {\Delta^n}$ and $E{\bf u}={\bf 0}$.
    Thus, ${\bf u}=E{\bf u}'=E^2{\bf u}'=E(E{\bf u}')=E{\bf u}={\bf 0}$ and so $\operatorname{Im}(E)\cap \operatorname{Ker}(E)=\{\bf 0\}$.
    We can conclude that $\Delta^n=\operatorname{Im}(E)\dot+\operatorname{Ker}(E)$.

    Let $\ell=\operatorname{rank}(E)$. We can find an ordered basis $({\bf v}_1,...,{\bf v_{\ell}})$ for $\operatorname{Im}(E)$ and an ordered basis $({\bf v}_{\ell+1},...,{\bf v}_n)$ for $\operatorname{Ker}(E)$.
    Then $({\bf v}_1,...,{\bf v}_n)$ is an ordered basis for $\Delta^n$ and so the matrix $A:=\left(\begin{array}{cccc}
         {\bf v}_1&{\bf v}_2&\cdots&{\bf v}_n   
    \end{array}\right)$ is an invertible matrix.
    Thus, we have $EA=AD_E$. Therefore, $E=AD_EA^{-1}$.
\end{proof}

Every idempotent in $M_n(\Delta)$ can be represented as follows.

\begin{corollary}
    \label{cor: all the idempotent of rank t}
    Let $0\le \ell\le n$. Then all the idempotent of rank $\ell$ can be constructed by
    \[
    A
    \left(
    \begin{array}{c|c}
     I_\ell&O  \\
     \hline
     O&O 
    \end{array}
    \right)A^{-1}
    \]
    where $A$ runs over all elements in $M_n(\Delta)^{\times}$.
\end{corollary}

\subsection{The Main Result}

Our main result is the following.

\begin{theorem}
    \label{thm: all the idempotent greater than E}
    Let $E\in {\mathscr I}(M_n(\Delta))$ and let $A$ be an invertible matrix as in {Proposition~\ref{diagonalization of idempotent matrix}} that
    \[
    E=A\left(\begin{array}{c|c}
    I_{\operatorname{rank}(E)} &O  \\
    \hline
    O &O 
    \end{array}\right)A^{-1}
    \]
    for some invertible matrix $A$.
    Then for $F\in \mathscr{I}(M_n(\Delta))$, we have $E\le F$ if and only if 
    \[
    F=A\left(\begin{array}{c|c}
    I_{\operatorname{rank}(E)} &O  \\
    \hline
    O &T 
    \end{array}\right)A^{-1}
    \]
    where $T\in {\mathscr I}(M_{n-\operatorname{rank}(E)}(\Delta))$.
    In particular, we have $\operatorname{rank}(E)\le \operatorname{rank}(F)$.
\end{theorem}
\begin{proof}
    ($\Rightarrow$) By {Proposition~\ref{diagonalization of idempotent matrix}}, there exists an invertible matrix $B$ such that $F=BD_FB^{-1}$ where
    \[
    D_F=\left(\begin{array}{c|c}
    I_{\operatorname{rank}(F)} &O  \\
    \hline
    O &O 
    \end{array}\right).
    \] 
    Because $EF=E=FE$, we have
    \[
    \left\{
    \begin{array}{l}
            AD_EA^{-1}BD_FB^{-1}=EF=E=AD_EA^{-1},  \\
            BD_FB^{-1}AD_EA^{-1}=FE=E=AD_EA^{-1}. 
    \end{array}
    \right.
    \]
    Then
    \[
    \left\{
    \begin{array}{ll}
            D_E(A^{-1}BD_FB^{-1}-A^{-1})&=O,  \\
            (BD_FB^{-1}A-A)D_E&=O. 
    \end{array}
    \right.
    \]
    Thus, by multiplying $A$ to the first equation and $A^{-1}$ to the second equation we have
    \[
    \left\{
    \begin{array}{ll}
            D_E(A^{-1}BD_FB^{-1}A-I_n)&=O,  \\
            (A^{-1}BD_FB^{-1}A-I_n)D_E&=O. 
    \end{array}
    \right.
    \]
    Because ${\bf e}_i^{\mathsf T}M$ is the $i$th row of the matrix $M$ and $M{\bf e}_i$ is the $i$th column of the matrix $M$, the equation $D_E(A^{-1}BD_FB^{-1}A-I_n)=O$ provides that the first $\operatorname{rank}(E)$ rows of $A^{-1}BD_FB^{-1}A-I_n$ are all zero, and the equation $(A^{-1}BD_FB^{-1}A-I_n)D_E=O$ tells us that the first $\operatorname{rank}(E)$ columns of $A^{-1}BD_FB^{-1}A-I_n$ are all zero.
    That is, we have
    \[
    A^{-1}BD_FB^{-1}A-I_n=
    \left(
    \begin{array}{c|c}
            O_{\operatorname{rank}(E)}&O  \\
            \hline
            O&* 
    \end{array}
    \right).
    \]
    Therefore,
    \[
    A^{-1}BD_FB^{-1}A=
    \left(
    \begin{array}{c|c}
            I_{\operatorname{rank}(E)}&O  \\
            \hline
            O&T 
    \end{array}
    \right),
    \]
    for some $T\in M_{n-\operatorname{rank}(E)}(\Delta)$.
    Since $D_F$ is an idempotent, we know that 
    \[
    (A^{-1}BD_FB^{-1}A)(A^{-1}BD_FB^{-1}A)=A^{-1}BD_FB^{-1}A
    \]
    and so $A^{-1}BD_FB^{-1}A$ is an idempotent. Then we have
    \[
    \left(
    \begin{array}{c|c}
            I_{\operatorname{rank}(E)}&O  \\
            \hline
            O&T 
    \end{array}
    \right)\left(
    \begin{array}{c|c}
            I_{\operatorname{rank}(E)}&O  \\
            \hline
            O&T 
    \end{array}
    \right)=\left(
    \begin{array}{c|c}
            I_{\operatorname{rank}(E)}&O  \\
            \hline
            O&T 
    \end{array}
    \right)
    \]
    and so
    \[
    \left(
    \begin{array}{c|c}
            I_{\operatorname{rank}(E)}&O  \\
            \hline
            O&T^2 
    \end{array}
    \right)
    =\left(
    \begin{array}{c|c}
            I_{\operatorname{rank}(E)}&O  \\
            \hline
            O&T 
    \end{array}
    \right).
    \]
    Therefore, $T$ is an idempotent.
    Finally,
    \[
    F=BD_FB^{-1}=A\left(
    \begin{array}{c|c}
            I_{\operatorname{rank}(E)}&O  \\
            \hline
            O&T 
    \end{array}
    \right)A^{-1},
    \]
    and we have $\operatorname{rank}(E)\le \operatorname{rank}(F)$.

    ($\Leftarrow$) By direct computation, we have
    \[
    A\left(\begin{array}{c|c}
    I_{\operatorname{rank}(E)} &O  \\
    \hline
    O &O 
    \end{array}\right)A^{-1}A\left(\begin{array}{c|c}
    I_{\operatorname{rank}(E)} &O  \\
    \hline
    O &T 
    \end{array}\right)A^{-1}=A\left(\begin{array}{c|c}
    I_{\operatorname{rank}(E)} &O  \\
    \hline
    O &O 
    \end{array}\right)A^{-1}
    \]
    and
    \[
    A\left(\begin{array}{c|c}
    I_{\operatorname{rank}(E)} &O  \\
    \hline
    O &T 
    \end{array}\right)A^{-1}A\left(\begin{array}{c|c}
    I_{\operatorname{rank}(E)} &O  \\
    \hline
    O &O 
    \end{array}\right)A^{-1}=A\left(\begin{array}{c|c}
    I_{\operatorname{rank}(E)} &O  \\
    \hline
    O &O 
    \end{array}\right)A^{-1}.
    \]
    Thus, $EF=E=FE$.
\end{proof}

\begin{remark}
    The subsequent content of this section aims at using {Theorem~\ref{thm: all the idempotent greater than E}} to construct the poset $({\mathscr I}(M_n(\Delta)),\le)$.
    By the result of this section, we can discuss basic combinatorial property on $({\mathscr I}(M_n({\mathbb F}_q)),\le)$.
    Nonetheless, we find very recently that \cite[Section~2]{CountingMatrixOverFiniteFields} construct the poset $({\mathscr I}(M_n({\mathbb F}_q)),\le)$ in a totally different approach.
\end{remark}

\begin{corollary}
     \label{cor: the construction of all the idempotent greater than E}
     Let $E\in {\mathscr I}(M_n(\Delta))$ and let $A$ be an invertible matrix as in {Proposition~\ref{diagonalization of idempotent matrix}} that
    \[
    E=A\left(\begin{array}{c|c}
    I_{\operatorname{rank}(E)} &O  \\
    \hline
    O &O 
    \end{array}\right)A^{-1}
    \]
    for some invertible matrix $A$. Then
    \[
    A\left(
    \begin{array}{c|c}
    I_{\operatorname{rank}(E)} &O  \\
    \hline
    O&
    B\left(
    \begin{array}{cc}
         I_{t} &O  \\
          O&O 
    \end{array}
    \right)B^{-1}
    \end{array}\right)A^{-1}
    \]
    for $0\le t\le n-\operatorname{rank}(E)$ and $B\in M_{n-\operatorname{rank}(E)}(\Delta)^{\times}$
    address all the elements in
    \[
    \{F\in {\mathscr I}(M_n(\Delta))\mid E\le F\}.
    \]
\end{corollary}
\begin{proof}
    Let $F\in {\mathscr I}(M_n(\Delta))$ with $E\le F$, then by {Theorem~\ref{thm: all the idempotent greater than E}} we know that
    \[
    F=A\left(\begin{array}{c|c}
    I_{\operatorname{rank}(E)} &O  \\
    \hline
    O &T 
    \end{array}\right)A^{-1}
    \]
    where $T\in {\mathscr I}(M_{n-\operatorname{rank}(E)}(\Delta))$.
    By {Corollary~\ref{cor: all the idempotent of rank t}}, we know that
    \[
    T=B\left(
    \begin{array}{cc}
         I_{t} &O  \\
          O&O 
    \end{array}
    \right)B^{-1}
    \]
    for some $B\in M_{n-\operatorname{rank}(E)}(\Delta)^{\times}$ and $t=\operatorname{rank}(T)\le n-\operatorname{rank}(E)$.
    The result follows.
\end{proof}

Let $E,F\in {\mathscr I}(M_n(\Delta))$ with $E\le F$.
If $\operatorname{rank}(F)=\operatorname{rank}(E)$, then by {Corollary~\ref{cor: the construction of all the idempotent greater than E}} $F=E$.
If $\operatorname{rank}(F)=\operatorname{rank}(E)+1$, then $F$ covers $E$.
Here comes a natural question. Does there exist $E,F\in {\mathscr I}(M_n(\Delta))$ such that $F$ covers $E$ with $\operatorname{rank}(F)>\operatorname{rank}(E)+1$?
The answer is negative. 
The only idempotents that cover $E$ are those with the rank increased by $1$.
We start from a lemma that reveals the existence.

\begin{lemma}
    \label{the existence of idempotent greater than E with rank greater by 1}
    Let $E$ in ${\mathscr I}(M_n(\Delta))$ with $\operatorname{rank}(E)<n$. There always exists a $F$ in ${\mathscr I}(M_n(\Delta))$ such that $\operatorname{rank}(F)=\operatorname{rank}(E)+1$. 
\end{lemma}
\begin{proof}
    Because $E\in {\mathscr I}(M_n(\Delta))$, by {Proposition~\ref{diagonalization of idempotent matrix}} we know that there exists an invertible matrix $A$ such that
    \[
    E=A\left(\begin{array}{c|c}
        I_{\operatorname{rank}(E)} & O \\
        \hline
        O & O
    \end{array}\right)A^{-1}.
    \]
    Let $T\in M_{n-\operatorname{rank}(E)}(\Delta)$ be the matrix defined by $T_{11}=1$ and $T_{ij}=0$ for all $(i,j)\neq (1,1)$. Then we know that $T^2=T$ is an idempotent and so $T\in {\mathscr I}(M_{n-\operatorname{rank}(E)}(\Delta))$.
    Let $F\in M_n(\Delta)$ be the matrix defined by
    \[
    F=A\left(\begin{array}{c|c}
        I_{\operatorname{rank}(E)} & O \\
        \hline
        O & T
    \end{array}\right)A^{-1}.
    \]
    Then $\operatorname{rank}(F)=\operatorname{rank}(E)+1$.
    Also, by {Theorem~\ref{thm: all the idempotent greater than E}}, we know that $E\le F$.
\end{proof}

\begin{proposition}
    \label{prop: idempotents that cover E}
    Let $E\le F$ in ${\mathscr I}(M_n(\Delta))$. Then $F$ covers $E$ if and only if $\operatorname{rank}(F)=\operatorname{rank}(E)+1$.
\end{proposition}
\begin{proof}
    ($\Leftarrow$) This direction is discussed in the paragraph before {Lemma~\ref{the existence of idempotent greater than E with rank greater by 1}}.

    ($\Rightarrow$) Suppose $F$ covers $E$ and $\operatorname{rank}(F)\gneq \operatorname{rank}(E)+1$.
    Because $E\le F$, by {Theorem~\ref{thm: all the idempotent greater than E}} we know that $E=AD_EA^{-1}$ and $F=A\left(\begin{array}{c|c}
        I_{\operatorname{rank}(E)} &O  \\
         \hline
        O &T 
    \end{array}\right)A^{-1}$ for some idempotent matrix $T\in M_{n-\operatorname{rank}(E)}(\Delta)$.
    Since $\operatorname{rank}(F)\gneq \operatorname{rank}(E)+1$, we have $\operatorname{rank}(T)\ge 2$.
    Write $T=BD_TB^{-1}$, where $D_T=\left(\begin{array}{c|c}
        I_{\operatorname{rank}(T)} &O  \\
         \hline
        O &O_{n-\operatorname{rank}(E)-\operatorname{rank}(T)} 
    \end{array}\right)$.
    Set
    \[
    G=A\left(
    \begin{array}{c|c}
        I_{\operatorname{rank}(E)} & O \\
        \hline
        O & B\left(\begin{array}{c|c}
            1 & O \\
            \hline
            O &O_{n-\operatorname{rank}(E)-1} 
        \end{array}\right)B^{-1}
    \end{array}
    \right)A^{-1}.
    \]
    We have
    \[
    F=A\left(
    \begin{array}{c|c}
        I_{\operatorname{rank}(E)} & O \\
        \hline
        O & B\left(\begin{array}{c|c}
            I_{\operatorname{rank}(T)} & O \\
            \hline
            O &O_{n-\operatorname{rank}(E)-\operatorname{rank}(T)} 
        \end{array}\right)B^{-1}
    \end{array}
    \right)A^{-1},
    \]
    \[
    G=A\left(
    \begin{array}{c|c}
        I_{\operatorname{rank}(E)} & O \\
        \hline
        O & B\left(\begin{array}{c|c}
            1 & O \\
            \hline
            O &O_{n-\operatorname{rank}(E)-1} 
        \end{array}\right)B^{-1}
    \end{array}
    \right)A^{-1},
    \]
    \[
    E=A\left(
    \begin{array}{c|c}
        I_{\operatorname{rank}(E)} & O \\
        \hline
        O & O_{n-\operatorname{rank}(E)}
    \end{array}
    \right)A^{-1}.
    \]
    Thus, $E\lneq G\lneq F$ and so $F$ doesn't cover $E$. Then there is a contradiction.
\end{proof}

By {Corollary~\ref{cor: all the idempotent of rank t}}, we know that the poset $({\mathscr I}(M_n(\Delta)),\le)$ can be layered depending on the rank of matrices.

\begin{corollary}
    \label{cor: the construction of idempotents that cover E}
    Let $0\le \ell\le n$. For $A\in M_n(\Delta)^{\times}$, the matrix
    \[
    A\left(
    \begin{array}{c|c}
         I_\ell &O  \\
         \hline
          O&O 
    \end{array}
    \right)A^{-1}
    \]
    is covered by
    \[
    A\left(
    \begin{array}{c|c}
        I_\ell & O \\
        \hline
        O & B\left(\begin{array}{c|c}
            1 & O \\
            \hline
            O &O_{n-\ell-1} 
        \end{array}\right)B^{-1}
    \end{array}
    \right)A^{-1}
    \]
    for $B$ runs over all invertible matrices of $M_{n-\ell}(\Delta)^{\times}$.
\end{corollary}
\begin{proof}
    By {Proposition~\ref{prop: idempotents that cover E}} and {Corollary~\ref{cor: the construction of all the idempotent greater than E}}.
\end{proof}

To construct the whole poset, we need to do {Corollary~\ref{cor: the construction of idempotents that cover E}} recursively.
We use $n=4$ to illustrate the process.
Note that if $E,F\in {\mathscr I}(M_n(\Delta))$ and $F$ covers $E$, then we write $E\lessdot F$ or $F\gtrdot E$.

\begin{example}
    Consider the case $n=4$.
    We have the following partial order relation.
    \[
    \left(
    \begin{array}{c|c}
        \begin{array}{ccc}
         1&&  \\
         &1&  \\
         &&1 
    \end{array} & O \\
    \hline
         O&A_1\left(1\right)A_1^{-1} 
    \end{array}
    \right)
    \gtrdot
    \left(
    \begin{array}{c|c}
        \begin{array}{ccc}
         1&&  \\
         &1&  \\
         &&1 
    \end{array} & O \\
    \hline
         O&0
    \end{array}
    \right)
    \]
    \[
    \left(
    \begin{array}{c|c}
        \begin{array}{cc}
         1&  \\
         &1  
    \end{array} & O \\
    \hline
         O&
         A_2\left(\begin{array}{cc}
             1 &  \\
              & 0
         \end{array}\right)A_2^{-1}
    \end{array}
    \right)
    \gtrdot
    \left(
    \begin{array}{c|c}
        \begin{array}{cc}
         1&  \\
         &1  
    \end{array} & O \\
    \hline
         O&
         \begin{array}{cc}
             0 &  \\
              & 0
         \end{array}
    \end{array}
    \right)
    \]
    \[
    \left(
    \begin{array}{c|c}
        1 & O \\
    \hline
         O&
         A_3\left(\begin{array}{ccc}
             1 &&  \\
              & 0& \\
              &&0
         \end{array}\right)A_3^{-1}
    \end{array}
    \right)
    \gtrdot
    \left(
    \begin{array}{c|c}
        1 & O \\
    \hline
         O&
         \begin{array}{ccc}
             0 &&  \\
              & 0& \\
              &&0
         \end{array}
    \end{array}
    \right)
    \]
    \[
    A_4\left(
    \begin{array}{cccc}
         1&&&  \\
         &0&&  \\
         &&0&  \\
         &&&0
    \end{array}
    \right)A_4^{-1}
    \gtrdot
    \left(
    \begin{array}{cccc}
         0&&&  \\
         &0&&  \\
         &&0&  \\
         &&&0
    \end{array}
    \right)
    \]
    for $A_i\in M_i(\Delta)^{\times}$.
    Combining above inequality, we have
    \[
    \begin{array}{ll}
        I_4 &\gtrdot\ \ \ A_4\left(
    \begin{array}{c|c}
        1 & O \\
        \hline
        O & A_3\left(\begin{array}{c|c}
            1 & O \\
            \hline
            O &A_2\left(\begin{array}{cc}
                1 &0  \\
                0 &0 
            \end{array}\right)A_2^{-1} 
        \end{array}\right)A_3^{-1}
    \end{array}
    \right)A_4^{-1}   \\
         & \gtrdot\ \ \ A_4\left(
    \begin{array}{c|c}
        1 & O \\
        \hline
        O & A_3\left(\begin{array}{ccc}
            1 & & \\
             &0&  \\
             &&0
        \end{array}\right)A_3^{-1}
    \end{array}
    \right)A_4^{-1}  \\
          &\gtrdot\ \ \ A_4\left(\begin{array}{cccc}
               1&&&  \\
               &0&&  \\
               &&0&  \\
               &&&0
          \end{array}\right)A_4^{-1} \\
          & \gtrdot\ \ \ O_4.
    \end{array}
    \]
    When $A_i$ runs over all element in $M_i(\Delta)^{\times}$, we obtain the complete poset.
    \hfill $\square$
\end{example}

Although this method is quite cumbersome, it allows us to gain deeper insights into the structure of the poset $({\mathscr I}(M_n(\Delta)),\le)$.
Note that $E\le F$ if and only if $I_n-F\le I_n-E$, so the poset is symmetric upside down.

Because $M_n({\mathbb F}_q)$ has a finite number of elements, with the above discussion, we can sketch this poset explicitly when $\Delta={\mathbb F}_q$ is a finite field.
To be more convenient, we define ${\mathscr I}^n_r(q)$\index{${\mathscr I}^n_r(q)$,${\mathscr I}^n_r$} be the number of $n\times n$ idempotent matrix in $M_n({\mathbb F}_q)$ of rank $r$. i.e.,
\[
{\mathscr I}^n_r(q):=|\{E\in {\mathscr I}(M_n({\mathbb F}_q))\mid \operatorname{rank}(E)=r\}|.
\]
In fact, the explicit formula for ${\mathscr I}^n_r(q)$, see \cite[page 60]{numofidempotent}, can be written as
\[
{\mathscr I}^n_r(q)=\begin{bmatrix}
     n  \\
     r
\end{bmatrix}_q\cdot q^{r(n-r)}
\]
where $\begin{bmatrix}
     n  \\
     r
\end{bmatrix}_q$ is the Gaussian binomial coefficient.
When the discussion is under the same $q$, we will abbreviate ${\mathscr I}^n_r(q)$ as ${\mathscr I}^n_r$.

\begin{example}
    Continue the previous example. Let $\Delta={\mathbb F}_q$ be a finite field and ${\mathscr I}^t_1$ be the number of $t\times t$ idempotent matrices of rank $1$.
    Because
    \[
    O_4\lessdot A_4\left(\begin{array}{cccc}
         1&&&  \\
         &0&&  \\
         &&0&  \\
         &&&0
    \end{array}\right)
    A_4^{-1},\ \forall\ A_4\in {\rm GL}_4({\mathbb F_q})
    \]
    and
    \[
    {\mathscr I}^4_1=\left|\left\{\left.A_4\left(\begin{array}{cccc}
         1&&&  \\
         &0&&  \\
         &&0&  \\
         &&&0
    \end{array}\right)
    A_4^{-1}\right\vert A_4\in {\rm GL}_4({\mathbb F}_q)\right\}\right|,
    \]
    we know that $O_4$ connects ${\mathscr I}^4_1$ lines to the first layer.
    Because
    \[
    A_4\left(
    \begin{array}{c|c}
        1 & O \\
    \hline
         O&
         A_3\left(\begin{array}{ccc}
             1 &&  \\
              & 0& \\
              &&0
         \end{array}\right)A_3^{-1}
    \end{array}
    \right)A_4^{-1}
    \gtrdot
    A_4\left(
    \begin{array}{c|c}
        1 & O \\
    \hline
         O&
         \begin{array}{ccc}
             0 &&  \\
              & 0& \\
              &&0
         \end{array}
    \end{array}
    \right)A_4^{-1}
    \]
    for all $A_3\in {\rm GL}_3({\mathbb F}_q)$ and
    \[
    {\mathscr I}^3_1=\left|\left\{\left.A_4\left(
    \begin{array}{c|c}
        1 & O \\
    \hline
         O&
         A_3\left(\begin{array}{ccc}
             1 &&  \\
              & 0& \\
              &&0
         \end{array}\right)A_3^{-1}
    \end{array}
    \right)A_4^{-1}\right\vert A_3\in {\rm GL}_3({\mathbb F}_q)\right\}\right|,
    \]
    we know that each element in layer $1$ connects ${\mathscr I}^3_1$ lines to the layer $2$.
    Because
    \begin{multline*}
        A_4\left(
    \begin{array}{c|c}
        1 & O \\
        \hline
        O & A_3\left(\begin{array}{c|c}
            1 & O \\
            \hline
            O &A_2\left(\begin{array}{cc}
                1 &0  \\
                0 &0 
            \end{array}\right)A_2^{-1} 
        \end{array}\right)A_3^{-1}
    \end{array}
    \right)A_4^{-1} \\
    \gtrdot A_4\left(
    \begin{array}{c|c}
        1 & O \\
        \hline
        O & A_3\left(\begin{array}{ccc}
            1 & & \\
             &0&  \\
             &&0
        \end{array}\right)A_3^{-1}
    \end{array}
    \right)A_4^{-1}
    \end{multline*}
    for all $A_2\in {\rm GL}_2({\mathbb F}_q)$ and ${\mathscr I}^2_1$ equals to
    \[
    \left|\left\{\left.A_4\left(
    \begin{array}{c|c}
        1 & O \\
        \hline
        O & A_3\left(\begin{array}{c|c}
            1 & O \\
            \hline
            O &A_2\left(\begin{array}{cc}
                1 &0  \\
                0 &0 
            \end{array}\right)A_2^{-1} 
        \end{array}\right)A_3^{-1}
    \end{array}
    \right)A_4^{-1}\right\vert A_2\in {\rm GL}_2({\mathbb F}_q)\right\}\right|,
    \]
    we know that each element in layer $2$ connects ${\mathscr I}^2_1$ lines to the layer $3$.
    Because
    \[
    I_4 \gtrdot A_4\left(
    \begin{array}{c|c}
        1 & O \\
        \hline
        O & A_3\left(\begin{array}{c|c}
            1 & O \\
            \hline
            O &A_2\left(\begin{array}{cc}
                1 &0  \\
                0 &0 
            \end{array}\right)A_2^{-1} 
        \end{array}\right)A_3^{-1}
    \end{array}
    \right)A_4^{-1},
    \]
    we know that each element in layer $3$ connects ${\mathscr I}^1_1=1$ line to the layer $4$ (the final layer).
    \hfill $\square$
\end{example}

\begin{example}
\label{exp: I(M_3(F_2))}
We give three explicit examples. Let $\Delta={\mathbb F}_2$. The following three Hasse diagrams are the posets of ${\mathscr I}(M_n({\mathbb F}_2))$ for $n=1,2,3$, respectively.

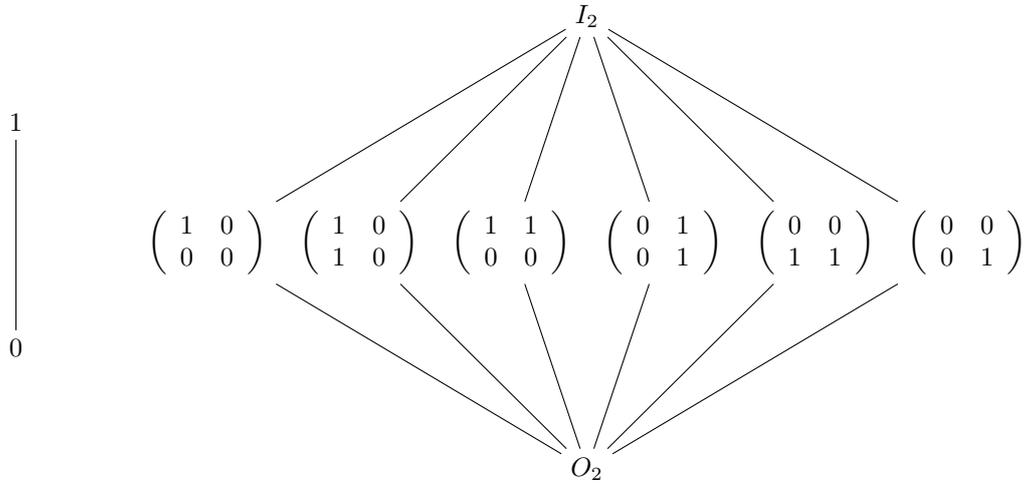
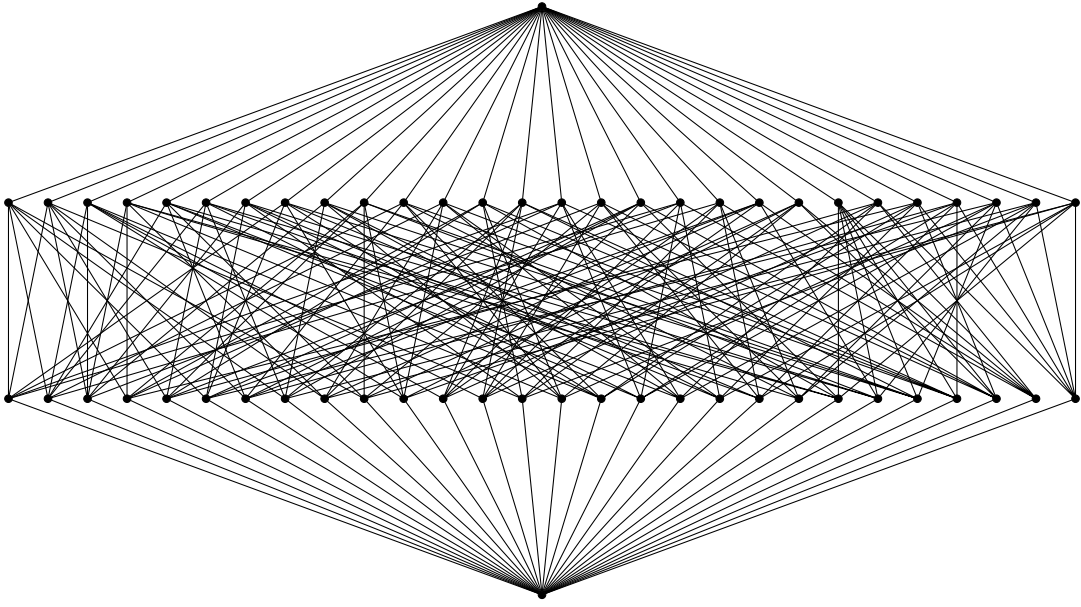
\[
\begin{tikzpicture}[scale=.5]
  \node (0)  at ( 0, -3) {$0$};
  \node (1)  at ( 0,  3)  {$1$};
  \node (2)  at ( 0,  -6.5)  {};
  \draw (0) -- (1);
\end{tikzpicture}
\qquad\qquad
\begin{tikzpicture}[scale=1]
  \node (0)  at ( 0, -3) {$O_2$};
  \node (1)  at ( 0,  3)  {$I_2$};
  \node (2)  at (-5,  0)  {$\left(\begin{array}{cc}
      1 & 0 \\
      0 & 0
  \end{array}\right)$};
  \node (3)  at (-3,  0)  {$\left(\begin{array}{cc}
      1 & 0 \\
      1 & 0
  \end{array}\right)$};
  \node (4)  at (-1,  0)  {$\left(\begin{array}{cc}
      1 & 1 \\
      0 & 0
  \end{array}\right)$};
  \node (5)  at ( 1,  0)  {$\left(\begin{array}{cc}
      0 & 1 \\
      0 & 1
  \end{array}\right)$};
  \node (6)  at ( 3,  0)  {$\left(\begin{array}{cc}
      0 & 0 \\
      1 & 1
  \end{array}\right)$};
  \node (7)  at ( 5,  0)  {$\left(\begin{array}{cc}
      0 & 0 \\
      0 & 1
  \end{array}\right)$};
  \draw (0) -- (2);
  \draw (0) -- (3);
  \draw (0) -- (4);
  \draw (0) -- (5);
  \draw (0) -- (6);
  \draw (0) -- (7);
  \draw (1) -- (2);
  \draw (1) -- (3);
  \draw (1) -- (4);
  \draw (1) -- (5);
  \draw (1) -- (6);
  \draw (1) -- (7);
\end{tikzpicture}
\]
\[
\begin{tikzpicture}[scale=.26]
  \node [mynode] (0)  at ( 0, -15) {};
  \node [mynode] (1)  at (-27,  -5)  {};
  \node [mynode] (2)  at (-25,  -5)  {};
  \node [mynode] (3)  at (-23,  -5)  {};
  \node [mynode] (4)  at (-21,  -5)  {};
  \node [mynode] (5)  at (-19,  -5)  {};
  \node [mynode] (6)  at (-17,  -5)  {};
  \node [mynode] (7)  at ( -27,  5)  {};
  \node [mynode] (8)  at (-15,  -5)  {};
  \node [mynode] (9)  at ( -25,  5)  {};
  \node [mynode] (10) at (-13,  -5)  {};
  \node [mynode] (11) at (-11,  -5)  {};
  \node [mynode] (12) at ( -9,  -5)  {};
  \node [mynode] (13) at ( -7,  -5)  {};
  \node [mynode] (14) at ( -5,  -5)  {};
  \node [mynode] (15) at ( -23,  5)  {};
  \node [mynode] (16) at ( -3,  -5)  {};
  \node [mynode] (17) at ( -1,  -5)  {};
  \node [mynode] (18) at (  1,  -5)  {};
  \node [mynode] (19) at (  3,  -5)  {};
  \node [mynode] (20) at ( -21,  5)  {};
  \node [mynode] (21) at ( -17,  5)  {};
  \node [mynode] (22) at (  5,  -5)  {};
  \node [mynode] (23) at ( -15,  5)  {};
  \node [mynode] (24) at ( -13,  5)  {};
  \node [mynode] (25) at ( -19,  5)  {};
  \node [mynode] (26) at (   1,  5)  {};
  \node [mynode] (27) at ( 0,  15)  {};
  \node [mynode] (28) at (   7,  5)  {};
  \node [mynode] (29) at (  -7,  5)  {};
  \node [mynode] (30) at (  7,  -5)  {};
  \node [mynode] (31) at ( -11,  5)  {};
  \node [mynode] (32) at (  -5,  5)  {};
  \node [mynode] (33) at (  9,  -5)  {};
  \node [mynode] (34) at ( 11,  -5)  {};
  \node [mynode] (35) at (   3,  5)  {};
  \node [mynode] (36) at (   9,  5)  {};
  \node [mynode] (37) at ( 13,  -5)  {};
  \node [mynode] (38) at ( 15,  -5)  {};
  \node [mynode] (39) at (  -3,  5)  {};
  \node [mynode] (40) at ( 17,  -5)  {};
  \node [mynode] (41) at (  -9,  5)  {};
  \node [mynode] (42) at (  11,  5)  {};
  \node [mynode] (43) at (   5,  5)  {};
  \node [mynode] (44) at ( 19,  -5)  {};
  \node [mynode] (45) at (  -1,  5)  {};
  \node [mynode] (46) at ( 21,  -5)  {};
  \node [mynode] (47) at ( 23,  -5)  {};
  \node [mynode] (48) at (  13,  5)  {};
  \node [mynode] (49) at ( 25,  -5)  {};
  \node [mynode] (50) at (  17,  5)  {};
  \node [mynode] (51) at (  19,  5)  {};
  \node [mynode] (52) at (  21,  5)  {};
  \node [mynode] (53) at (  23,  5)  {};
  \node [mynode] (54) at (  25,  5)  {};
  \node [mynode] (55) at (  27,  5)  {};
  \node [mynode] (56) at (  15,  5)  {};
  \node [mynode] (57) at ( 27,  -5)  {};
  \draw (0) -- (1);
  \draw (0) -- (2);
  \draw (0) -- (3);
  \draw (0) -- (4);
  \draw (0) -- (5);
  \draw (0) -- (6);
  \draw (0) -- (8);
  \draw (0) -- (10);
  \draw (0) -- (11);
  \draw (0) -- (12);
  \draw (0) -- (13);
  \draw (0) -- (14);
  \draw (0) -- (16);
  \draw (0) -- (17);
  \draw (0) -- (18);
  \draw (0) -- (19);
  \draw (0) -- (22);
  \draw (0) -- (30);
  \draw (0) -- (33);
  \draw (0) -- (34);
  \draw (0) -- (37);
  \draw (0) -- (38);
  \draw (0) -- (40);
  \draw (0) -- (44);
  \draw (0) -- (46);
  \draw (0) -- (47);
  \draw (0) -- (49);
  \draw (0) -- (57);
  \draw (27) -- (7);
  \draw (27) -- (9);
  \draw (27) -- (15);
  \draw (27) -- (20);
  \draw (27) -- (25);
  \draw (27) -- (21);
  \draw (27) -- (23);
  \draw (27) -- (24);
  \draw (27) -- (31);
  \draw (27) -- (41);
  \draw (27) -- (29);
  \draw (27) -- (32);
  \draw (27) -- (39);
  \draw (27) -- (45);
  \draw (27) -- (26);
  \draw (27) -- (35);
  \draw (27) -- (43);
  \draw (27) -- (28);
  \draw (27) -- (36);
  \draw (27) -- (42);
  \draw (27) -- (48);
  \draw (27) -- (56);
  \draw (27) -- (50);
  \draw (27) -- (51);
  \draw (27) -- (52);
  \draw (27) -- (53);
  \draw (27) -- (54);
  \draw (27) -- (55);
  \draw (1) -- (7);
  \draw (1) -- (9);
  \draw (1) -- (24);
  \draw (1) -- (29);
  \draw (1) -- (26);
  \draw (1) -- (28);
  \draw (2) -- (7);
  \draw (2) -- (15);
  \draw (2) -- (25);
  \draw (2) -- (21);
  \draw (2) -- (32);
  \draw (2) -- (36);
  \draw (3) -- (9);
  \draw (3) -- (15);
  \draw (3) -- (20);
  \draw (3) -- (23);
  \draw (3) -- (31);
  \draw (3) -- (35);
  \draw (4) -- (7);
  \draw (4) -- (20);
  \draw (4) -- (41);
  \draw (4) -- (39);
  \draw (4) -- (26);
  \draw (4) -- (51);
  \draw (5) -- (9);
  \draw (5) -- (21);
  \draw (5) -- (41);
  \draw (5) -- (39);
  \draw (5) -- (28);
  \draw (5) -- (50);
  \draw (6) -- (7);
  \draw (6) -- (23);
  \draw (6) -- (24);
  \draw (6) -- (43);
  \draw (6) -- (42);
  \draw (6) -- (53);
  \draw (8) -- (9);
  \draw (8) -- (25);
  \draw (8) -- (29);
  \draw (8) -- (43);
  \draw (8) -- (42);
  \draw (8) -- (52);
  \draw (10) -- (7);
  \draw (10) -- (21);
  \draw (10) -- (31);
  \draw (10) -- (45);
  \draw (10) -- (43);
  \draw (10) -- (54);
  \draw (11) -- (9)
        (11) -- (20)
        (11) -- (32)
        (11) -- (45)
        (11) -- (42)
        (11) -- (55);
  \draw (12) -- (7)
        (12) -- (25)
        (12) -- (41)
        (12) -- (35)
        (12) -- (48)
        (12) -- (55);
  \draw (13) -- (15)
        (13) -- (24)
        (13) -- (41)
        (13) -- (32)
        (13) -- (48)
        (13) -- (52);
  \draw (14) -- (15)
        (14) -- (45)
        (14) -- (26)
        (14) -- (43)
        (14) -- (36)
        (14) -- (50);
  \draw (16) -- (15)
        (16) -- (45)
        (16) -- (35)
        (16) -- (28)
        (16) -- (42)
        (16) -- (51);
  \draw (17) -- (9)
        (17) -- (23)
        (17) -- (39)
        (17) -- (36)
        (17) -- (48)
        (17) -- (54);
  \draw (18) -- (15)
        (18) -- (31)
        (18) -- (29)
        (18) -- (39)
        (18) -- (48)
        (18) -- (53);
  \draw (19) -- (20)
        (19) -- (21)
        (19) -- (24)
        (19) -- (29)
        (19) -- (50)
        (19) -- (51);
  \draw (22) -- (25)
        (22) -- (23)
        (22) -- (26)
        (22) -- (28)
        (22) -- (52)
        (22) -- (53);
  \draw (30) -- (25)
        (30) -- (31)
        (30) -- (32)
        (30) -- (26)
        (30) -- (50)
        (30) -- (54);
  \draw (33) -- (23)
        (33) -- (31)
        (33) -- (32)
        (33) -- (28)
        (33) -- (51)
        (33) -- (55);
  \draw (34) -- (21)
        (34) -- (24)
        (34) -- (35)
        (34) -- (36)
        (34) -- (52)
        (34) -- (55);
  \draw (37) -- (20)
        (37) -- (29)
        (37) -- (35)
        (37) -- (36)
        (37) -- (53)
        (37) -- (54);
  \draw (38) -- (20)
        (38) -- (25)
        (38) -- (39)
        (38) -- (43)
        (38) -- (56)
        (38) -- (55);
  \draw (40) -- (21)
        (40) -- (23)
        (40) -- (41)
        (40) -- (42)
        (40) -- (56)
        (40) -- (54);
  \draw (44) -- (24)
        (44) -- (31)
        (44) -- (41)
        (44) -- (45)
        (44) -- (56)
        (44) -- (53);
  \draw (46) -- (29)
        (46) -- (32)
        (46) -- (39)
        (46) -- (45)
        (46) -- (56)
        (46) -- (52);
  \draw (47) -- (26)
        (47) -- (35)
        (47) -- (43)
        (47) -- (48)
        (47) -- (56)
        (47) -- (51);
  \draw (49) -- (28)
        (49) -- (36)
        (49) -- (42)
        (49) -- (48)
        (49) -- (56)
        (49) -- (50);
  \draw (57) -- (50)
        (57) -- (51)
        (57) -- (52)
        (57) -- (53)
        (57) -- (54)
        (57) -- (55);
\end{tikzpicture}
\]
\hfill $\square$
\end{example}

By observing the Hasse diagram of ${\mathscr I}(M_n({\mathbb F}_2))$, we found that the poset of ${\mathscr I}(M_3({\mathbb F}_2))$ contains many copies of the poset of ${\mathscr I}(M_2({\mathbb F}_2))$.
To describe this phenomenon, we provide the following definition.
The whole picture will be stated in {Theorem~\ref{thm: interval isomorphic to smaller poset}}.

\begin{definition}
    Let ${\mathcal P}_1=(X_1,\preceq_1)$ and ${\mathcal P}_2=(X_2,\preceq_2)$ be two posets.
    An {\bf order isomorphism}\index{order isomorphism} from ${\mathcal P}_1$ to ${\mathcal P}_2$ is a bijection $f:X_1\to X_2$ satisfying the following property: for all $x,y\in X_1$, $x\preceq_1 y$ if and only if $f(x)\preceq_2 f(y)$.
    If there is an order isomorphism between ${\mathcal P}_1$ and ${\mathcal P}_2$, then we say ${\mathcal P}_1$ and ${\mathcal P}_2$ are {\bf isomorphic} and write ${\mathcal P}_1\simeq {\mathcal P}_2$.
\end{definition}

For abbreviation, we may write $X_1\simeq X_2$ to stand for ${\mathcal P}_1\simeq {\mathcal P}_2$. To prove {Theorem~\ref{thm: interval isomorphic to smaller poset}}, we need a bunch of lemmas. 

\begin{lemma}
    \label{lem: conjugation is an order isomorphism}
    For any $A\in M_n(\Delta)^{\times}$, the conjugation map $\phi_A:{\mathscr I}(M_n(\Delta))\to {\mathscr I}(M_n(\Delta))$ defined by $\phi_A(E)=AEA^{-1}$ for any $E\in {\mathscr I}(M_n(\Delta))$ is an order isomorphism. 
    Furthermore, if $X$ is a subset of ${\mathscr I}(M_n(\Delta))$, then $\phi_A(X)\simeq X$.
\end{lemma}
\begin{proof}
    Let $E\in {\mathscr I}(M_n(\Delta))$, then the conjugation $AEA^{-1}\in {\mathscr I}(M_n(\Delta))$. Thus, the map $\phi_A$ is well-defined.
    If $E\le F$, then we have $AEA^{-1}\le AFA^{-1}$ and so $\phi_A(E)\le \phi_A(F)$.
    Also, it is easy to see that the conjugation map $\phi_A$ is a bijection.
    Therefore, $\phi_A$ is an order isomorphism.
\end{proof}

\begin{lemma}
    \label{lem: identify idempotent matrix in smaller form}
    Let $0\le \ell\le n$.
    Let $E\in M_n(\Delta)$ and assume
    \[
    E=A\left(\begin{array}{c|c}
       I_\ell  & O  \\
       \hline
        O & S 
    \end{array}\right)
    A^{-1}
    \]
    for some invertible matrix $A$ and $S\in M_{n-\ell}(\Delta)$.
    Then $S$ is uniquely determined by $A$.
    Moreover, we have $E\in {\mathscr I}(M_n(\Delta))$ if and only if $S\in {\mathscr I}(M_{n-\ell}(\Delta))$.
\end{lemma}
\begin{proof}
    Suppose
    \[
    A\left(\begin{array}{c|c}
       I_\ell  & O  \\
       \hline
        O & S 
    \end{array}\right)
    A^{-1}
    =E=
    A\left(\begin{array}{c|c}
       I_\ell  & O  \\
       \hline
        O & S' 
    \end{array}\right)
    A^{-1},
    \]
    then
    \[
    \left(\begin{array}{c|c}
       I_\ell  & O  \\
       \hline
        O & S 
    \end{array}\right)
    =A^{-1}EA=
    \left(\begin{array}{c|c}
       I_\ell  & O  \\
       \hline
        O & S' 
    \end{array}\right).
    \]
    Thus, $S=S'$.
    
    From the block multiplication, we know that
    \[
    E^2=A\left(\begin{array}{c|c}
       I_\ell  & O  \\
       \hline
        O & S 
    \end{array}\right)
    A^{-1}
    A\left(\begin{array}{c|c}
       I_\ell  & O  \\
       \hline
        O & S 
    \end{array}\right)
    A^{-1}
    =A\left(\begin{array}{c|c}
       I_\ell  & O  \\
       \hline
        O & S^2 
    \end{array}\right)
    A^{-1}.
    \]
    Thus, $E^2=E$ if and only if $S^2=S$.
    The result follows.
\end{proof}

With the notation in {Lemma~\ref{lem: identify idempotent matrix in smaller form}}, the ordering of $E\in {\mathscr I}(M_n(\Delta))$ is determined by the ordering of $S\in {\mathscr I}(M_{n-\ell}(\Delta))$. That is,

\begin{lemma}
    \label{lem: E<F iff S<T}
    Let $0\le \ell\le n$.
    Let $E,F$ be matrices in ${\mathscr I}(M_n(\Delta))$ which can be written as
    \[
    E=A\left(\begin{array}{c|c}
       I_\ell  & O  \\
       \hline
        O & S 
    \end{array}\right)
    A^{-1},\ 
    F=A\left(\begin{array}{c|c}
       I_\ell  & O  \\
       \hline
        O & T 
    \end{array}\right)
    A^{-1}
    \]
    for some invertible matrix $A$ and $S,T\in {\mathscr I}(M_{n-\ell}(\Delta))$ that are uniquely determined in {Lemma~\ref{lem: identify idempotent matrix in smaller form}}.
    Then $E\le F$ if and only if $S\le T$.
\end{lemma}
\begin{proof}
    From the block multiplication, we know that
    \[
    EF=A\left(\begin{array}{c|c}
       I_\ell  & O  \\
       \hline
        O & S 
    \end{array}\right)
    A^{-1}
    A\left(\begin{array}{c|c}
       I_\ell  & O  \\
       \hline
        O & T 
    \end{array}\right)
    A^{-1}
    =A\left(\begin{array}{c|c}
       I_\ell  & O  \\
       \hline
        O & ST 
    \end{array}\right)
    A^{-1}
    \]
    and
    \[
    FE=A\left(\begin{array}{c|c}
       I_\ell  & O  \\
       \hline
        O & S 
    \end{array}\right)
    A^{-1}
    A\left(\begin{array}{c|c}
       I_\ell  & O  \\
       \hline
        O & T 
    \end{array}\right)
    A^{-1}
    =
    A\left(\begin{array}{c|c}
       I_\ell  & O  \\
       \hline
        O & TS 
    \end{array}\right)
    A^{-1}.
    \]
    Therefore, $E\le F$ if and only if $S\le T$.
\end{proof}

\begin{lemma}
    \label{lem: D_t is isomorphic to I_t}
    Let $0\le \ell\le n$ and $D_\ell\in {\mathscr I}(M_n(\Delta))$ with the form
    \[
    D_\ell=\left(\begin{array}{c|c}
        I_\ell & O \\
        \hline
        O & O_{n-\ell}
    \end{array}\right).
    \]
    Then we have $[O_n,D_\ell]\simeq [O_\ell,I_\ell]={\mathscr I}(M_\ell(\Delta))$.
\end{lemma}
\begin{proof}
    Let $H\in [O_n,D_\ell]$ i.e., $O_n\le H\le D_\ell$, then $HD_\ell=H=D_\ell H$.
    Because $H=HD_\ell$, we know that ${\bf h}_i={\bf 0}$ for $\ell<i\le n$ where ${\bf h}_i$ is the $i$th column of $H$.
    Because $H=D_\ell H$, we know that $_i{\bf h}={\bf 0}^{\mathsf T}$ for $\ell<i\le n$ where $_i{\bf h}$ is the $i$th row of $H$.
    Thus,
    \[
    H=\left(\begin{array}{c|c}
        V & O  \\
        \hline
        O & O
    \end{array}\right)
    \]
    for some $V\in M_\ell(\Delta)$.
    Because $H^2=H$,
    \[
    \left(\begin{array}{c|c}
        V & O  \\
        \hline
        O & O
    \end{array}\right)=H=H^2=
    \left(\begin{array}{c|c}
        V & O  \\
        \hline
        O & O
    \end{array}\right)\left(\begin{array}{c|c}
        V & O  \\
        \hline
        O & O
    \end{array}\right)=
    \left(\begin{array}{c|c}
        V^2 & O  \\
        \hline
        O & O
    \end{array}\right)
    \]
    and so $V^2=V$ i.e., $V\in {\mathscr I}(M_\ell(\Delta))$.
    What we have shown is that for every $H\in [O_n,D_\ell]$, there exists a unique idempotent $V\in [O_\ell,I_\ell]$ such that
    \[
    H=\left(\begin{array}{c|c}
        V & O  \\
        \hline
        O & O
    \end{array}\right).
    \]

    Define $f:[O_n,D_\ell]\to [O_\ell,I_\ell]$ by $f(H)=V$. Then $f$ is a well-defined function by the above argument.
    Let $H_1,H_2\in [O_n,D_\ell]$ with $f(H_1)=f(H_2)$.
    We know that
    \[
    H_1=\left(\begin{array}{c|c}
        V_1 & O  \\
        \hline
        O & O
    \end{array}\right),\ 
    H_2=\left(\begin{array}{c|c}
        V_2 & O  \\
        \hline
        O & O
    \end{array}\right)
    \]
    and $V_1=f(H_1)=f(H_2)=V_2$. Thus, $H_1=H_2$ and so $f$ is injective.
    Let $V'\in [O_\ell,I_\ell]$ and define
    \[
    H'=\left(\begin{array}{c|c}
        V' & O  \\
        \hline
        O & O
    \end{array}\right).
    \]
    Then $H'\in [O_n,D_\ell]$ and $f(H')=V'$. Thus, $f$ is surjective.
    Combining above, we know that $f:[O_n,D_\ell]\to [O_\ell,I_\ell]$ is a bijection.

    We remaining to show that $f$ preserve the partial order.
    Let $X,Y\in [O_n,D_\ell]$. Then by the above discussion, we have
    \[
    X=\left(\begin{array}{c|c}
        f(X) & O  \\
        \hline
        O & O
    \end{array}\right),\ 
    Y=\left(\begin{array}{c|c}
        f(Y) & O  \\
        \hline
        O & O
    \end{array}\right).
    \]
    Because
    \[
    XY=\left(\begin{array}{c|c}
        f(X) & O  \\
        \hline
        O & O
    \end{array}\right)\left(\begin{array}{c|c}
        f(Y) & O  \\
        \hline
        O & O
    \end{array}\right)=
    \left(\begin{array}{c|c}
        f(X)f(Y) & O  \\
        \hline
        O & O
    \end{array}\right),
    \]
    and
    \[
    YX=\left(\begin{array}{c|c}
        f(Y) & O  \\
        \hline
        O & O
    \end{array}\right)\left(\begin{array}{c|c}
        f(X) & O  \\
        \hline
        O & O
    \end{array}\right)=
    \left(\begin{array}{c|c}
        f(Y)f(X) & O  \\
        \hline
        O & O
    \end{array}\right),
    \]
    we know that $X\le Y$ if and only if $f(X)\le f(Y)$.
    Therefore, $f$ is an order isomorphism and $[O_n,D_\ell]\simeq [O_\ell,I_\ell]$.
\end{proof}

In fact, we have following result.

\begin{theorem}
    \label{thm: interval isomorphic to smaller poset}
    Let $E\le F\in {\mathscr I}(M_n(\Delta))$. Then the subposet $[E,F]$ is order isomorphic to ${\mathscr I}(M_{\operatorname{rank}(F)-\operatorname{rank}(E)}(\Delta))$.
\end{theorem}
\begin{proof}
    Because $E$ is an idempotent,
    \[
    E=A\left(\begin{array}{c|c}
       I_{\operatorname{rank}(E)}  & O  \\
       \hline
        O &O 
    \end{array}\right)
    A^{-1}
    \]
    for some invertible matrix $A$.
    Because $E\le F$, by {Theorem~\ref{thm: all the idempotent greater than E}} we have
    \[
    F=A\left(\begin{array}{c|c}
       I_{\operatorname{rank}(E)}  & O  \\
       \hline
        O &T 
    \end{array}\right)
    A^{-1}
    \]
    where $T\in {\mathscr I}(M_{n-\operatorname{rank}(E)}(\Delta))$ is uniquely determined by {Lemma~\ref{lem: identify idempotent matrix in smaller form}}.
    We claim that $[E,F]\simeq [O,T]$. Let $G\in [E,F]$, then $E\le G\le F$. By {Theorem~\ref{thm: all the idempotent greater than E}} and {Lemma~\ref{lem: identify idempotent matrix in smaller form}}, we know that there exists a unique $S\in {\mathscr I}(M_{n-\operatorname{rank}(E)}(\Delta))$ such that
    \[
    G=A\left(\begin{array}{c|c}
       I_{\operatorname{rank}(E)}  & O  \\
       \hline
        O &S 
    \end{array}\right)
    A^{-1}
    \]
    and $S\le T$ i.e., $S\in [O,T]$.
    Define $f_A:[E,F]\to [O,T]$ by $f_A(G)=S$.
    Because the uniqueness of $S$, we know that the function $f_A$ is well-defined.
    Let $G_1,G_2\in [E,F]$ with $f_A(G_1)=f_A(G_2)$. We know that
    \[
    G_1=A\left(\begin{array}{c|c}
       I_{\operatorname{rank}(E)}  & O  \\
       \hline
        O &S_1 
    \end{array}\right)
    A^{-1},\ 
    G_2=A\left(\begin{array}{c|c}
       I_{\operatorname{rank}(E)}  & O  \\
       \hline
        O &S_2 
    \end{array}\right)
    A^{-1}
    \]
    and $S_1=f_A(G_1)=f_A(G_2)=S_2$. Thus, $G_1=G_2$ and so $f_A$ is injective.
    Let $S'\in [O,T]$ and define
    \[
    G'=A\left(\begin{array}{c|c}
       I_{\operatorname{rank}(E)}  & O  \\
       \hline
        O &S' 
    \end{array}\right)
    A^{-1}.
    \]
    Then by {Lemma~\ref{lem: E<F iff S<T}} we know that $G'\in [E,F]$ and $f_A(G')=S'$. Thus, $f_A$ is surjective.
    Combining above, we know that $f_A:[E,F]\to [O,T]$ is a bijection.
    We also need to show that $f_A$ preserve the order.
    Let $X,Y\in [E,F]$. Then by the above discussion, we have
    \[
    X=A\left(\begin{array}{c|c}
       I_{\operatorname{rank}(E)}  & O  \\
       \hline
        O &f_A(X) 
    \end{array}\right)
    A^{-1},\ 
    Y=A\left(\begin{array}{c|c}
       I_{\operatorname{rank}(E)}  & O  \\
       \hline
        O &f_A(Y)
    \end{array}\right)
    A^{-1}.
    \]
    By {Lemma~\ref{lem: E<F iff S<T}}, we know that $X\le Y$ if and only if $f_A(X)\le f_A(Y)$.
    Therefore, $f_A$ is an order isomorphism and $[E,F]\simeq [O,T]$.
    
    Because $T$ is an idempotent,
    \[
    T=
    B\left(\begin{array}{c|c}
       I_{\operatorname{rank}(T)}  & O  \\
       \hline
        O &O 
    \end{array}\right)
    B^{-1}
    \]
    for some invertible matrix $B$.
    Consider the conjugation $\phi_{B^{-1}}$. It is not hard to see that
    \[
    \phi_{B^{-1}}([O,T])=
    \left[O,\left(\begin{array}{c|c}
       I_{\operatorname{rank}(T)}  & O  \\
       \hline
        O &O 
    \end{array}\right)\right].
    \]
    Thus, by {Lemma~\ref{lem: conjugation is an order isomorphism}} we obtain
    \[
    [O,T]\simeq\left[O,\left(\begin{array}{c|c}
       I_{\operatorname{rank}(T)}  & O  \\
       \hline
        O &O 
    \end{array}\right)\right].
    \]
    Write $D_T=\left(\begin{array}{c|c}
       I_{\operatorname{rank}(T)}  & O  \\
       \hline
        O &O 
    \end{array}\right)$, then we have $[E,F]\simeq [O,D_T]$.
    By {Lemma~\ref{lem: D_t is isomorphic to I_t}}, we know that $[O,D_T]\simeq [O,I_{\operatorname{rank}(T)}]$.
    
    As $\operatorname{rank}(T)=\operatorname{rank}(F)-\operatorname{rank}(E)$, we have $[O,I_{\operatorname{rank}(T)}]={\mathscr I}(M_{\operatorname{rank}(F)-\operatorname{rank}(E)}(\Delta))$ and the proof is completed.
\end{proof}

\section{Constructing Idempotent Matrices}
\label{sec: Constructing Idempotent Matrices}

As we mentioned before that every idempotent matrix $E\in M_n(K)$ can be obtained by $E=ADA^{-1}$ where $D={\rm diag}(1,...,1,0,...,0)$ for some invertible matrix $A$, however, computing the inverse of a matrix is quite tedious.
Also, the above method often yields many identical idempotent matrices.
Hence, we aim to develop new methods for constructing elements of this kind.

\subsection{Idempotents in Matrix Ring over a UFD}
\label{subsec: Idempotents in matrix ring over a UFD}

Let $R$ be an integral domain.
In this section, we want to characterize the idempotent elements in $M_2(R)$.

Let $K$ be the field of fractions of the integral domain $R$.
Then $R$ can be seen as a subring of $K$.
The rank of a matrix $A\in M_2(R)$ is defined to be the rank of a matrix $A\in M_2(K)$.

Let $E$ be an idempotent in $M_2(R)$. Note that $E$ can also be viewed as an element in $M_2(K)$. It is easy to see that if $\operatorname{rank}(E)=0$, then $E=O$.
If $\operatorname{rank}(E)=2$, then $E$ is invertible in $M_2(K)$ and $E^2=E$ and hence $E=I_2$.

\begin{lemma}
    \label{lem: idempotent in M2(integral domain)}
    Let $R$ be an integral domain.
    Let $E\in M_2(R)$ with $\operatorname{rank}(E)=1$.
    Then $E$ is an idempotent if and only if
    \[
    E=\left(\begin{array}{cc}
        a & b \\
        c & 1-a
    \end{array}\right)\ {\textrm with}\ a(1-a)=bc
    \]
    for $a,b,c\in R$.
\end{lemma}
\begin{proof}
    View $E$ as an element in $M_2(K)$ where $K$ is the field of fractions of $R$.
    Let $m(x)$ be the minimal polynomial of $E$.
    Then $E^2=E$ if and only if $m(x)\mid x^2-x$.
    If $\deg(m(x))=1$, then $m(x)=x$ (and so $E=O$) or $m(x)=x-1$ (and so $E=I_2$), a contradiction.
    If $\deg(m(x))=2$, then the characteristic polynomial of $E$ is $m(x)=x^2-x$.
    Therefore, $E$ is an idempotent of rank $1$ if and only if $m(x)=x^2-x$.
    Note that $m(x)=x^2-\operatorname{tr}(E)x+\det(E)$.
    We can conclude that $E$ is an idempotent of rank $1$ if and only if ${\rm tr}(E)=1$ and $\det(E)=0$.
\end{proof}

If the ring $R$ is a unique factorization domain (UFD), then we can improve the statement. That is, each idempotent $E\in M_2(R)$ of $\operatorname{rank}(E)=1$ has a better form.

\begin{proposition}
    \label{prop: idempotent in M2(UFD)}
    Let $R$ be a UFD.
    Let $E\in M_2(R)$ with $\operatorname{rank}(E)=1$.
    Then $E$ is an idempotent in $M_2(R)$ if and only if
    \[
    E=\left(\begin{array}{cc}
         sa & sb \\
         ta & tb
    \end{array}\right)
    \]
    where $a,b,s,t\in R$ satisfying $sa+tb=1$ i.e., $\operatorname{tr}(E)=1$.
\end{proposition}
\begin{proof}
    ($\Rightarrow$)
    Let $K$ be the field of fractions of $R$.
    Because we can view $E\in M_2(K)$ and $\operatorname{rank}(E)=1$, we know that two row vectors of matrix $E$ are parallel. Thus, by {Lemma~\ref{lem: idempotent in M2(integral domain)}}, either
    \[
    E=\left(\begin{array}{cc}
         0  & 0 \\
         x & 1
    \end{array}\right),\ {\rm or}\ 
    E=\left(\begin{array}{cc}
         c  & d \\
         kc & kd
    \end{array}\right)
    \]
    for some $k\in K$ and $x,c,d,kc,kd\in R$ with $(c,d)\neq (0,0)$ and $kd=1-c$.
    It is easy to see that the former matrix satisfies our conclusion. We only need to discuss the later matrix.
        
    If $k=0$, then $0=kd=1-c$ and so $c=1$. Thus,
    \[
    E=\left(\begin{array}{cc}
         1  & d \\
         0 & 0
    \end{array}\right)
    \]
    satisfies our conclusion.
    
    If $k\neq 0$, we can write $k=t/s$ with $t,s\in R\setminus\{0\}$ and $\gcd(t,s)=1$.
    Since $tc/s=kc$ and $td/s=kd\in R$, we have $s\mid tc$ and $s\mid td$.
    Because $R$ is a UFD, we know that every irreducible element is a prime element. 
    Write $s$ into a product of irreducible elements.
    Since $s\mid tc$ and $s\mid td$, we can conclude that $s\mid c$ and $s\mid d$.
    Thus, $c=sa$ and $d=sb$ for some $a,b\in R$.
    Hence, we obtain that
    \[
    E=\left(\begin{array}{cc}
         sa  & sb \\
         ta & tb
    \end{array}\right)
    \]
    and $sa+tb=c+kd=c+(1-c)=1$.

    ($\Leftarrow$) By direct computation.
\end{proof}

\begin{example}
    Let $R={\mathbb Z}[\sqrt{5}{\bf i}]$ where ${\bf i}=\sqrt{-1}$. Then $R$ is not a UFD since $6$ can be factorized as
    \[
    2\cdot 3=6=(1+\sqrt{5}{\bf i})\cdot(1-\sqrt{5}{\bf i})
    \]
    and $2,3,1+\sqrt{5}{\bf i},1-\sqrt{5}{\bf i}$ are all irreducible elements in ${\mathbb Z}[\sqrt{5}{\bf i}]$.
    By taking $a=-2$ in {Lemma~\ref{lem: idempotent in M2(integral domain)}}, we find an idempotent
    \[
    E=\left(\begin{array}{cc}
        -2 & 1+\sqrt{5}{\bf i} \\
        -1+\sqrt{5}{\bf i} & 3
    \end{array}\right)\in M_2({\mathbb Z}[\sqrt{5}{\bf i}]).
    \]
    However, $E$ cannot be written as the form in {Proposition~\ref{prop: idempotent in M2(UFD)}}.
    Suppose that this is not the case i.e.,
    \[
    \left(\begin{array}{cc}
        -2 & 1+\sqrt{5}{\bf i} \\
        -1+\sqrt{5}{\bf i} & 3
    \end{array}\right)=
    \left(\begin{array}{cc}
        sa & sb \\
        ta & tb
    \end{array}\right).
    \]
    Then $-2=sa, 1+\sqrt{5}{\bf i}=sb, -1+{\sqrt 5}{\bf i}=ta, 3=tb$.
    Note that ${\mathbb Z}[{\sqrt 5}{\bf i}]^{\times}=\{1,-1\}$.
    Because $-2$ is an irreducible element in ${\mathbb Z}[\sqrt{5}{\bf i}]$, we know that one of $s$ or $a$ must be unit.
    If $s$ is a unit, then either $(s,a)=(1,-2)$ or $(s,a)=(-1,2)$.
    Since $-1+\sqrt{5}{\bf i}=ta$, we have $2\mid -1+\sqrt{5}{\bf i}$, a contradiction.
    If $a$ is a unit, then either $(s,a)=(2,-1)$ or $(s,a)=(-2,1)$.
    Since $1+\sqrt{5}{\bf i}=sb$, we have $2\mid 1+\sqrt{5}{\bf i}$, a contradiction.
    \hfill $\square$
\end{example}

In fact, the form in {Proposition~\ref{prop: idempotent in M2(UFD)}} still holds for idempotents in $E\in M_n(R)$ of $\operatorname{rank}(E)=1$.
We write $A_{(i,j)}$ to be the matrix formed by deleting the $i$th row and $j$th column of $A$.

\begin{corollary}
    \label{cor: idempotent in Mn(UFD) of rank 1}
    Let $R$ be a UFD. Let $E\in M_n(R)$ with $\operatorname{rank}(E)=1$. Then $E$ is an idempotent in $M_n(R)$ if and only if
    \[
    E=\left(\begin{array}{cccc}
         s_1a_1&s_1a_2&\cdots&s_1a_n  \\
         s_2a_1&s_2a_2&\cdots&s_2a_n  \\
         \vdots&\vdots&\ddots&\vdots  \\
         s_na_1&s_na_2&\cdots&s_na_n
    \end{array}\right)
    \]
    where $s_i,a_j\in R$ satisfying $\sum_{i=1}^{n}s_ia_i=1$ i.e., ${\rm tr}(E)=1$.
\end{corollary}
\begin{proof}
    ($\Rightarrow$) We use induction on $n$. The case $n=2$ is proven in {Proposition~\ref{prop: idempotent in M2(UFD)}}.
    Let $E\in M_n(R)$ be an idempotent. Suppose $E$ has a zero row (say the $\ell$th row for some $1\le \ell\le n$).
    Because $E^2=E$, we have $(E_{(\ell,\ell)})^2=E_{(\ell,\ell)}$.
    To see this, write $E=(e_{ij})$ and $E_{(\ell,\ell)}=(e_{ij}^{(\ell)})$. Then
    \begin{enumerate}[(a)]
        \item If $1\le i,j< \ell\le n-1$, then $e_{ij}^{(\ell)}=e_{ij}$.
        \item If $1\le i< \ell\le j\le n-1$, then $e_{ij}^{(\ell)}=e_{i(j+1)}$
        \item If $1\le j< \ell\le i\le n-1$, then $e_{ij}^{(\ell)}=e_{(i+1)j}$.
        \item If $1\le \ell\le i,j\le n-1$, then $e_{ij}^{(\ell)}=e_{(i+1)(j+1)}$.
    \end{enumerate}
    Since $E=E^2$ and $e_{\ell j}=0$ for all $j$, we know that
    \[
    e_{ij}=\sum_{k=1}^{n}e_{ik}e_{kj}=\sum_{k=1}^{\ell-1}e_{ik}e_{kj}+\sum_{k=\ell+1}^{n}e_{ik}e_{kj}.
    \]
    Thus, by (a) to (d), we have
    \[
    e_{ij}^{(\ell)}=\sum_{k=1}^{n-1}e_{ik}^{(\ell)}e_{kj}^{(\ell)}.
    \]
    Therefore, $E_{(\ell,\ell)}=(E_{(\ell,\ell)})^2$. 
    Because $E_{(\ell,\ell)}$ is an $(n-1)\times (n-1)$ idempotent matrix, by induction hypothesis we may write
    \[
    e_{ij}^{(\ell)}=\left\{\begin{array}{ll}
         s_ia_i& {\textrm {if}}\ 1\le i,j< \ell\le n-1, \\
         s_ia_{j+1}& {\textrm {if}}\ 1\le i< \ell\le j\le n-1, \\
         s_{i+1}a_j& {\textrm {if}}\ 1\le j< \ell\le i\le n-1, \\
         s_{i+1}a_{j+1}&{\textrm {if}}\ 1\le \ell\le i,j\le n-1,
    \end{array}\right.
    \]
    where $\sum_{k=1}^{\ell-1}s_ka_k+\sum_{k=\ell}^{n-1}s_{k+1}a_{k+1}=1$ and $s_i,a_i\in R$ with $i\neq \ell$.
    Let ${\bf e}_\ell$ be the $\ell$th column of $E$.
    Because $\operatorname{rank}(E)=1$, ${\bf e}_\ell$ must be of the form
    \[
    {\bf e}_\ell=(s_1a_\ell,s_2a_\ell,...,s_na_\ell)^{\mathsf T}\in R^n,
    \]
    where $a_\ell\in K$. We need to show that $a_\ell\in R$.
    From $\sum_{k=1}^{\ell-1}s_ka_k+\sum_{k=\ell}^{n-1}s_{k+1}a_{k+1}=1$, we know that there exists an index $1\le k\le n$ and $k\neq \ell$ such that $s_k\neq 0$.
    Since $s_k\mid s_ka_\ell\in R$, $s_ka_\ell=s_kt$ for some $t\in R$.
    Because $R$ is an integral domain, we know that $a_\ell=t\in R$.
    On the other hand, there is an index $1\le m\le n$ such that $a_m\neq 0$.
    Since $0=e_{\ell m}=s_\ell a_m$, we can conclude that $s_{\ell}=0$.
    Thus, $\sum_{k=1}^{n}s_ka_k=1$ and so $E$ can be written in the form we wanted.

    Now, suppose every row of $E$ is nonzero. Because $\operatorname{rank}(E)=1$, we can write
    \[
    E=\left(\begin{array}{cccc}
         c_1&c_2&\cdots&c_n  \\
         k_2c_1&k_2c_2&\cdots&k_2c_n  \\
         \vdots&\vdots&\ddots&\vdots  \\
         k_nc_1&k_nc_2&\cdots&k_nc_n
    \end{array}\right),
    \]
    where $c_1,...,c_n\in R$ and $k_2,...,k_n\in K$.
    Write $k_i=s_i/t_i$ for some $s_i,t_i\in R$ and $\gcd(s_i,t_i)=1$ ($k_1:=1$ and $s_1=t_1:=1$).
    Then $s_ic_j/t_i=k_ic_j\in R$ for all $i,j$ and so $t_i\mid s_ic_j$.
    Because $R$ is a UFD, by the same argument in {Proposition~\ref{prop: idempotent in M2(UFD)}}, we know that $t_i\mid c_j$. Thus, $c_j=t_ia_j$ for some $a_j\in R$.
    Therefore, $E$ has the form
    \[
    E=\left(\begin{array}{cccc}
         s_1a_1&s_1a_2&\cdots&s_1a_n  \\
         s_2a_1&s_2a_2&\cdots&s_2a_n  \\
         \vdots&\vdots&\ddots&\vdots  \\
         s_na_1&s_na_2&\cdots&s_na_n
    \end{array}\right).
    \]
    Because $E$ is an idempotent matrix of rank $1$, we know that ${\rm tr}(E)=1$ (view $E\in M_n(K)$ and the characteristic polynomial is $x^n-x^{n-1}$).
    Hence, we can conclude the result.

    ($\Leftarrow$) By direct computation.
\end{proof}

\begin{corollary}
    \label{cor: idempotent in M3(UFD)}
    Let $R$ be a UFD. Let $E\in M_3(R)$. Then 
    \begin{enumerate}[(a)]
        \item $E$ is an idempotent with $\operatorname{rank}(E)=1$ if and only if
        \[
        E=\left(\begin{array}{ccc}
             sa&sb&sc  \\
             ta&tb&tc  \\
             ua&ub&uc
        \end{array}\right)
        \]
        where $a,b,c,s,t,u\in R$ satisfying $sa+tb+uc=1$.
        \item $E$ is an idempotent with $\operatorname{rank}(E)=2$ if and only if
        \[
        E=\left(\begin{array}{rrr}
             1-sa&-sb&-sc  \\
             -ta&1-tb&-tc  \\
             -ua&-ub&1-uc
        \end{array}\right)
        \]
        where $a,b,c,s,t,u\in R$ satisfying $sa+tb+uc=1$.
    \end{enumerate}
\end{corollary}
\begin{proof}
    (a) is immediately follows from {Corollary~\ref{cor: idempotent in Mn(UFD) of rank 1}}.
    For (b) with $\operatorname{rank}(E)=2$, we can derive that $I_3-E$ is an idempotent matrix of rank $1$ by {Proposition~\ref{diagonalization of idempotent matrix}}. Note that we set $\Delta$ to be the ring of fraction of $R$ in {Proposition~\ref{diagonalization of idempotent matrix}}.
\end{proof}

\subsection{Idempotents in Matrix Ring over a PID}
\label{subsec: Idempotents in matrix ring over a PID}

When $R$ is a principal ideal domain (PID), we have other tools to deal with the problem more deeply.
Let $A\in M_n(R)$. Then there exist invertible matrices $P,Q$ such that
\[
PAQ=\left(\begin{array}{ccccccc}
     d_1&&&&&&  \\
     &d_2&&&&&  \\
     &&\ddots&&&&  \\
     &&&d_\ell&&&  \\
     &&&&0&&  \\
     &&&&&\ddots&  \\
     &&&&&&0
\end{array}\right)
\]
where $d_i\mid d_{i+1}$. The diagonal matrix is called the {\bf Smith normal form of} $A$\index{Smith normal form} where $d_i$ is unique up to a unit in $R$.
For the Smith normal form, we take \cite[Section~III.8 to Section~III.10]{LecturesinAbstractAlgebraLinearAlgebra} as a reference.

\begin{definition}
    Let $L$ be a field and $V$ be an $n$-dimensional vector space over $L$.
    Let $\beta=\{{\bf v}_1,...,{\bf v}_n\}$ be a basis for $V$.
    Let $R$ be a ring contained in $L$.
    An $R$-{\bf lattice} $\Lambda\subseteq V$ is an $R$-module of $V$ (as the usual scalar multiplication in $V$) of the form
    \[
    \Lambda=R{\bf v}_1+\cdots+R{\bf v}_n=\left\{\sum_{i=1}^{n}r_i{\bf v}_i\ \middle|\  r_i\in R\right\}.
    \]
    When $V=L^n$ with ${\bf v}_i={\bf e}_i$, we write $\Lambda$ to be $\Lambda_0$ i.e., $\Lambda_0=R^n$.
\end{definition}

\begin{lemma}
    \label{lem: lattice quotient is isomorphic when the denominator is act by an invertible matrix}
    Let $R$ be an integral domain and $\Lambda$ be an $R$-lattice with $\Lambda\subseteq \Lambda_0$.
    Let $P\in M_n(R)^{\times}$. Define
    \[
    P\Lambda:=\{P{\bf u}\mid {\bf u}\in \Lambda\}.
    \]
    Then we have
    \[
    \Lambda_0/P\Lambda\simeq \Lambda_0/\Lambda
    \]
    as $R$-modules.
\end{lemma}
\begin{proof}
    Define $\psi:\Lambda_0\to \Lambda_0/P\Lambda$ by $\psi({\bf v})=P{\bf v}+P\Lambda$.
    Then $\psi$ is an $R$-module homomorphism.
    Let ${\bf u}+P\Lambda\in \Lambda_0/P\Lambda$.
    Because $P$ is invertible, $P^{-1}{\bf u}\in \Lambda_0$. Thus, $\psi(P^{-1}{\bf u})=PP^{-1}{\bf u}+P\Lambda={\bf u}+P\Lambda$ and so $\psi$ is an onto map.
    Clearly, by the definition of $\psi$, $\Lambda$ is contained in ${\rm Ker}(\psi)$.
    Let ${\bf u}\in {\rm Ker}(\psi)$. Then $\psi({\bf u})= P\Lambda$ and so $P{\bf u}\in P\Lambda$.
    Thus, $P{\bf u}=P{\bf w}$ for some ${\bf w}\in \Lambda$.
    Since $P$ is invertible, we can conclude that ${\bf u}={\bf w}\in \Lambda$ and so ${\rm Ker}(\psi)\subseteq \Lambda$.
    Hence, ${\rm Ker}(\psi)=\Lambda$ and so $\Lambda_0/\Lambda\simeq\Lambda_0/P\Lambda$.
\end{proof}

Applying {Lemma~\ref{lem: lattice quotient is isomorphic when the denominator is act by an invertible matrix}} to the beginning matrix $A$, we can conclude that
\[
R^n/AR^n=\Lambda_0/P^{-1}DQ^{-1}\Lambda_0=\Lambda_0/P^{-1}(D\Lambda_0)\simeq \Lambda_0/D\Lambda_0= R^n/DR^n
\]
where $D={\rm diag}(d_1,...,d_\ell,0,...,0)$, $d_i \neq 0$, since $D=PAQ$ and $P,Q\in M_n(R)^{\times}$.
Because $DR^n=\{D{\bf v}\mid {\bf v}\in R^n\}=\left\{\sum_{i=1}^{\ell}r_id_i{\bf e}_i\ \middle|\ r_i\in R\right\}=d_1R\oplus\cdots\oplus d_\ell R\oplus (0)\oplus\cdots\oplus (0)$, we have
\[
R^n/DR^n\simeq R/d_1R\oplus \cdots \oplus R/d_\ell R\oplus R/(0)\oplus\cdots\oplus R/(0).
\]
Thus, $R^n/AR^n\simeq R/d_1R\oplus \cdots \oplus R/d_\ell R\oplus R^{n-\ell}$.

Assume that $A$ is an idempotent and $R$ is a PID. Consider the endomorphism of $R^n$ by multiplying $A$ on elements of $R^n$. Then $R^n$ is a direst sum of ${\rm Im}(A)$ and ${\rm Ker}(A)$, each of which is a free submodule. In particular, $R^n / A R^n = R^n / {\rm Im}(A) \simeq {\rm Ker}(A)$ is free, and thus torsion-free. Thus, each $d_i$ must be a unit.
Therefore, the Smith normal form of $A$ can be adjusted to the following
\[
PAQ=\left(\begin{array}{ccccccc}
     1&&&&&&  \\
     &1&&&&&  \\
     &&\ddots&&&&  \\
     &&&1&&&  \\
     &&&&0&&  \\
     &&&&&\ddots&  \\
     &&&&&&0
\end{array}\right).
\]
By writing $S=P^{-1}$ and $T=Q^{-1}$, we have
\[
A=S\left(\begin{array}{c|c}
    I_\ell & O \\
    \hline
    O & O
\end{array}\right)T.
\]
Because $A^2=A$, we have
\[
\left(\begin{array}{c|c}
    I_\ell & O \\
    \hline
    O & O
\end{array}\right)TS\left(\begin{array}{c|c}
    I_\ell & O \\
    \hline
    O & O
\end{array}\right)=
\left(\begin{array}{c|c}
    I_\ell & O \\
    \hline
    O & O
\end{array}\right)
\]
and so
\[
\left(\begin{array}{ccc}
    \rotvert &{}_1{\bf t}& \rotvert \\
     &\vdots&  \\
    \rotvert &{}_\ell{\bf t}& \rotvert \\
    \rotvert &{\bf 0}^{\mathsf T}& \rotvert \\
     &\vdots&  \\
    \rotvert &{\bf 0}^{\mathsf T}& \rotvert 
\end{array}\right)
\left(\begin{array}{cccccc}
     \vert&&\vert&\vert&&\vert  \\
     {\bf s}_1&\cdots&{\bf s}_\ell&{\bf 0}&\cdots&{\bf 0}  \\
     \vert&&\vert&\vert&&\vert 
\end{array}\right)=
\left(\begin{array}{c|c}
    I_\ell & O \\
    \hline
    O & O
\end{array}\right)
\]
where ${}_i{\bf t}$ is the $i$th row of the matrix $T$ and ${\bf s}_i$ is the $i$th column of the matrix $S$ for $1\le i\le \ell$.

With the above discussion and replacing $A$ by $E$ as our convention, we can conclude the following result.

\begin{theorem}
    \label{thm: the form of idempotent matrix with invertible S T}
    Let $R$ be a PID and $E\in M_n(R)$. Then $E$ is an idempotent of rank $\ell$ if and only if there exist $S,T\in M_n(R)^{\times}$ such that
    \[
\left(\begin{array}{ccc}
    \rotvert &{}_1{\bf t}& \rotvert \\
     &\vdots&  \\
    \rotvert &{}_\ell{\bf t}& \rotvert \\
    \rotvert &{\bf 0}^{\mathsf T}& \rotvert \\
     &\vdots&  \\
    \rotvert &{\bf 0}^{\mathsf T}& \rotvert 
\end{array}\right)
\left(\begin{array}{cccccc}
     \vert&&\vert&\vert&&\vert  \\
     {\bf s}_1&\cdots&{\bf s}_\ell&{\bf 0}&\cdots&{\bf 0}  \\
     \vert&&\vert&\vert&&\vert 
\end{array}\right)=
\left(\begin{array}{c|c}
    I_\ell & O \\
    \hline
    O & O
\end{array}\right)
\]
and
\[
E=S\left(\begin{array}{c|c}
    I_\ell & O \\
    \hline
    O & O
\end{array}\right)T.
\]
\end{theorem}

If we write
\[
\left(\begin{array}{ccc}
    \rotvert &{}_1{\bf t}& \rotvert \\
     &\vdots&  \\
    \rotvert &{}_\ell{\bf t}& \rotvert \\
    \rotvert &{\bf 0}^{\mathsf T}& \rotvert \\
     &\vdots&  \\
    \rotvert &{\bf 0}^{\mathsf T}& \rotvert 
\end{array}\right)=
\left(\begin{array}{c|c}
    A_{\ell\times \ell} & B_{\ell\times(n-\ell)} \\
    \hline
    O_{(n-\ell)\times \ell} & O_{(n-\ell)\times (n-\ell)}
\end{array}\right)
\]
and
\[
\left(\begin{array}{cccccc}
     \vert&&\vert&\vert&&\vert  \\
     {\bf s}_1&\cdots&{\bf s}_\ell&{\bf 0}&\cdots&{\bf 0}  \\
     \vert&&\vert&\vert&&\vert 
\end{array}\right)=
\left(\begin{array}{c|c}
    C_{\ell\times \ell} & O_{\ell\times(n-\ell)} \\
    \hline
    D_{(n-\ell)\times \ell} & O_{(n-\ell)\times (n-\ell)}
\end{array}\right)
\]
then the equation tells us that
\[
\left(\begin{array}{c|c}
    A & B \\
     \hline
    O & O
\end{array}\right)\left(\begin{array}{c|c}
    C & O \\
     \hline
    D & O
\end{array}\right)=
\left(\begin{array}{c|c}
    I_\ell & O \\
     \hline
    O & O
\end{array}\right)
\]
and so
\[
\left(\begin{array}{c|c}
    AC+BD & O \\
     \hline
    O & O
\end{array}\right)=\left(\begin{array}{c|c}
    I_\ell & O \\
     \hline
    O & O
\end{array}\right).
\]
In other words, $AC+BD=I_\ell$. In this situation, we know that
\[
\begin{array}{ll}
    E &=\left(\begin{array}{c|c}
    C & O \\
     \hline
    D & O
\end{array}\right)\left(\begin{array}{c|c}
    I_\ell & O \\
     \hline
    O & O
\end{array}\right)\left(\begin{array}{c|c}
    A & B \\
     \hline
    O & O
\end{array}\right)=
\left(\begin{array}{c|c}
    C & O \\
     \hline
    D & O
\end{array}\right)\left(\begin{array}{c|c}
    A & B \\
     \hline
    O & O
\end{array}\right)  \\
     & =
\left(\begin{array}{c|c}
    CA & CB \\
     \hline
    DA & DB
\end{array}\right)
\end{array}.
\]
Note that
\[
\begin{array}{ll}
    E^2 & =\left(\begin{array}{c|c}
    CA & CB \\
     \hline
    DA & DB
\end{array}\right)\left(\begin{array}{c|c}
    CA & CB \\
     \hline
    DA & DB
\end{array}\right) \\
     & =
\left(\begin{array}{c|c}
    C(AC+BD)A & C(AC+BD)B \\
     \hline
    D(AC+BD)A & D(AC+BD)B
\end{array}\right) \\
     & =\left(\begin{array}{c|c}
    CA & CB \\
     \hline
    DA & DB
\end{array}\right)=E
\end{array}.
\]

Therefore, we can drop the invertible condition on $S$ and $T$.
We have the following.

\begin{theorem}
    \label{thm: idempotent in Mn(PID)}
    Let $R$ be a PID and $E\in M_n(R)$.
    Then $E$ is an idempotent of rank $\ell$ if and only if there exist
    \[
    S=\left(\begin{array}{c|c}
        C & O \\
        \hline
        D & O
    \end{array}\right),\ {\rm and}\ 
    T=\left(\begin{array}{c|c}
        A & B \\
        \hline
        O & O
    \end{array}\right)
    \]
    with $A,C\in M_\ell(R)$ and $AC+BD=I_\ell$ such that
    \[
    E=S\left(\begin{array}{c|c}
    I_\ell & O \\
     \hline
    O & O
\end{array}\right)T=ST.
    \]
\end{theorem}
\begin{proof}
    ($\Rightarrow$) By {Theorem~\ref{thm: the form of idempotent matrix with invertible S T}}, this part is clear.

    ($\Leftarrow$) The only thing we need to worry about is whether $\operatorname{rank}(E)=\ell$ is correct.
    By assumption, we know that
    \[
    TS=\left(\begin{array}{c|c}
        A & B \\
        \hline
        O & O
    \end{array}\right)
    \left(\begin{array}{c|c}
        C & O \\
        \hline
        D & O
    \end{array}\right)
    =
    \left(\begin{array}{c|c}
        AC+BD & O \\
        \hline
        O & O
    \end{array}\right)
    =\left(\begin{array}{c|c}
        I_{\ell} & O \\
        \hline
        O & O
    \end{array}\right)
    \]
    and so $\operatorname{rank}(TS)=\ell$.
    Since 
    \[
    (TS)^{\mathsf T}=\left(\begin{array}{c|c}
        I_{\ell} & O \\
        \hline
        O & O
    \end{array}\right)^{\mathsf T}=\left(\begin{array}{c|c}
        I_{\ell} & O \\
        \hline
        O & O
    \end{array}\right)=TS,
    \]
    we have $\operatorname{rank}(S^{\mathsf T}T^{\mathsf T})=\ell$.
    By elementary linear algebra, we know that $\operatorname{rank}(TS)\le \operatorname{rank}(T)$ and $\operatorname{rank}(S^{\mathsf T}T^{\mathsf T})\le \operatorname{rank}(S^{\mathsf T})=\operatorname{rank}(S)$.
    Thus, we have $\ell=\operatorname{rank}(TS)\le \operatorname{rank}(T)\le \ell$ and $\ell=\operatorname{rank}(S^{\mathsf T}T^{\mathsf T})\le \operatorname{rank}(S)\le \ell$. We can conclude that $\operatorname{rank}(T)=\ell=\operatorname{rank}(S)$.
    Therefore, all nonzero columns of $S$ are linearly independent over $K$ and all nonzero rows of $T$ are linearly independent over $K$ where $K$ is the field of fractions of $R$.
    Thus, we can find $S',T'\in M_n(K)^{\times}$ satisfying
    \[
    T'=\left(\begin{array}{c|c}
        A & B \\
        \hline
        G & H
    \end{array}\right)\ {\textrm {and}}\ 
    S'=\left(\begin{array}{c|c}
        C & U \\
        \hline
        D & V
    \end{array}\right).
    \]
    By direct computation,
    \[
    S'\left(\begin{array}{c|c}
    I_\ell & O \\
     \hline
    O & O
\end{array}\right)T'=\left(\begin{array}{c|c}
        C & U \\
        \hline
        D & V
    \end{array}\right)\left(\begin{array}{c|c}
    I_\ell & O \\
     \hline
    O & O
\end{array}\right)\left(\begin{array}{c|c}
        A & B \\
        \hline
        G & H
    \end{array}\right)=
    S\left(\begin{array}{c|c}
    I_\ell & O \\
     \hline
    O & O
\end{array}\right)T=E.
    \]
    Since $S'$ and $T'$ are invertible, we obtain
    \[
    \operatorname{rank}(E)=\operatorname{rank}(S'\left(\begin{array}{c|c}
    I_\ell & O \\
     \hline
    O & O
\end{array}\right)T')=\operatorname{rank}(\left(\begin{array}{c|c}
    I_\ell & O \\
     \hline
    O & O
\end{array}\right))=\ell.
    \]
\end{proof}

Observing the above formula
\[
E=ST=\left(\begin{array}{c|c}
    CA & CB \\
     \hline
    DA & DB
\end{array}\right),
\]
we found that the expression of $E$ is very similar to what we write in {Proposition~\ref{prop: idempotent in M2(UFD)}}.
Note that $B,D$ may not be square matrices.
The above theorem can be summarized in the following alternative form.

\begin{theorem}
    \label{thm: idempotent in Mn(PID) in alternative form}
    Let $R$ be a PID and $E\in M_n(R)$. Then $E$ is an idempotent of rank $\ell$ if and only if
    \[
    E=\left(\begin{array}{c|c}
        CA & CB \\
         \hline
        DA & DB
    \end{array}\right)
    \]
    for some $A,C\in M_{\ell}(R)$ and matrices $B,D$ such that $AC+BD=I_{\ell}$.
\end{theorem}

\begin{example}
    Let $R$ be a PID.
    Let $a_i,b_i\in R$ with $\gcd(a_i,b_i)=1$. Because $R$ is a PID, there exists $g_i,h_i\in R$ such that $a_ig_i+b_ih_i=1$.
    By choosing
    \[
    A=\left(\begin{array}{cc}
        a_1 & b_1 \\
        0 & 0
    \end{array}\right),\ 
    B=\left(\begin{array}{cc}
        0 & 0 \\
        a_2 & b_2
    \end{array}\right),\ 
    C=\left(\begin{array}{cc}
        g_1 & -b_1 \\
        h_1 & a_1
    \end{array}\right),\ 
    D=\left(\begin{array}{cc}
        b_2 & g_2 \\
        -a_2 & h_2
    \end{array}\right),
    \]
    then
    \[
    E=\left(\begin{array}{cccc}
         a_1g_1&b_1g_1&-a_2b_1&-b_2b_1  \\
         a_1h_1&b_1h_1&a_2a_1&b_2a_1  \\
         a_1b_2&b_1b_2&a_2g_2&b_2g_2  \\
         -a_1a_2&-b_1a_2&a_2h_2&b_2h_2
    \end{array}\right)
    \]
    is an idempotent.
    We give three concrete examples as $R={\mathbb Z}$, $R={\mathbb F}_2[x]$ and $R={\mathbb Q}[x]$, respectively.
    \begin{enumerate}[(a)]
        \item Let $R={\mathbb Z}$. By choosing
    \[
    A=\left(\begin{array}{cc}
        3 & -1 \\
        0 & 0
    \end{array}\right),\ 
    B=\left(\begin{array}{cc}
        0 & 0 \\
        -3 & 7
    \end{array}\right),\ 
    C=\left(\begin{array}{cc}
        2 & 1 \\
        5 & 3
    \end{array}\right),\ 
    D=\left(\begin{array}{cc}
        7 & -5 \\
        3 & -2
    \end{array}\right),
    \]
    then we can construct an idempotent
    \[
    E=\left(\begin{array}{cccc}
         6&-2&-3&7  \\
         15&-5&-9&21  \\
         21&-7&15&-35  \\
         9&-3&6&-14
    \end{array}\right)\in M_4({\mathbb Z}).
    \]
    
        \item Let $R={\mathbb F}_2[x]$. By choosing
    \[
    A=\left(\begin{array}{cc}
        x^2 & x^2+x+1 \\
        0 & 0
    \end{array}\right),\ 
    B=\left(\begin{array}{cc}
        0 & 0 \\
        x^3+x+1 & x+1
    \end{array}\right),
    \]
    \[
    C=\left(\begin{array}{cc}
        x & x^2+x+1 \\
        x+1 & x^2
    \end{array}\right),\ 
    D=\left(\begin{array}{cc}
        x+1 & 1 \\
        x^3+x+1 & x^2+x
    \end{array}\right),
    \]
    then we can construct an idempotent
    \[
    E=\left(\begin{array}{cccc}
         x^3&x^3+x^2+x&x^5+x^4+1&x^3+1  \\
         x^3+x^2&x^3+1&x^5+x^3+x^2&x^3+x^2  \\
         x^3+x^2&x^3+1&x^3+x+1&x+1  \\
         x^5+x^3+x^2&x^5+x^4+1&x^5+x^4+x^3+x&x^3+x
    \end{array}\right)\in M_4({\mathbb F}_2[x]).
    \]
    Note that $1=-1\in {\mathbb F}_2[x]$.

    \item Let $R={\mathbb Q}[x]$.
    By choosing $a_1=x+1,b_1=x^4+x^3+x^2+x+1,a_2=x^2+x+1,b_2=x^2+1$ and $g_1=-x^3-x,h_1=1,g_2=-x,h_2=1+x$, then we can construct an idempotent
    {
    \footnotesize
    \[
    \left(\begin{array}{cccc}
         \begin{array}{r}
              -x^4-x^3-x^2  \\
               -x
         \end{array}&\begin{array}{r}
              -x^7-x^6-2x^5  \\
               -2x^4-2x^3\\
               -x^2-x
         \end{array}&\begin{array}{r}
              -x^6-2x^5-3x^4  \\
               -3x^3-3x^2\\
               -2x-1
         \end{array}& \begin{array}{r}
              -x^6-x^5-2x^4  \\
               -2x^3-2x^2\\
               -x-1
         \end{array} \\
         x+1&\begin{array}{r}
              x^4+x^3+x^2  \\
              +x+1
         \end{array}&\begin{array}{r}
              x^3+2x^2+2x  \\
              +1
         \end{array}&\begin{array}{r}
              x^3+x^2+x  \\
              +1
         \end{array}  \\
         \begin{array}{r}
              x^3+x^2+x  \\
               +1
         \end{array}&\begin{array}{r}
              x^6+x^5+2x^4  \\
              +2x^3+2x^2  \\
              +x+1
         \end{array}&-x^3-x^2-x&-x^3-x  \\
         \begin{array}{r}
              -x^3-2x^2-2x  \\
               -1
         \end{array}&\begin{array}{r}
              -x^6-2x^5-3x^4  \\
               -3x^3-3x^2\\
               -2x-1
         \end{array}&\begin{array}{r}
              x^3+2x^2+2x  \\
               +1
         \end{array}&\begin{array}{r}
              x^3+x^2+x  \\
               +1
         \end{array}
    \end{array}\right).
    \]
    }
    \end{enumerate}
    Because the idempotents we have usually found containing a lot of zero entries, the above three idempotents without zero entries are somewhat unexpected.
    \hfill $\square$
\end{example}

\subsection{Additional Constructions}
\label{subsec: Additional Constructions}

Let $K$ be a field.
Let $A\in M_m(K)$ and $B\in M_n(K)$. The {\bf Kronecker product of} $A$ {\bf and} $B$\index{Kronecker product} is denoted as $A\boxtimes B$\index{$\boxtimes$: Kronecker product} and defined by
\[
A\boxtimes B=\left(\begin{array}{cccc}
     a_{11}B&a_{12}B&\cdots&a_{1m}B  \\
     a_{21}B&a_{22}B&\cdots&a_{2m}B \\
     \vdots&\vdots&\ddots&\vdots \\
     a_{m1}B&a_{m2}B&\cdots&a_{mm}B
\end{array}\right)\in M_{mn}(K),
\]
where $A=(a_{ij})$.
The Kronecker product is $K$-bilinear. That is,

\begin{lemma}
    \label{lem: Kronecker product is bilinear}
    Let $A,A'\in M_m(K)$ and $B,B'\in M_n(K)$. Then we have
    \[
    \begin{array}{ll}
        (A+A')\boxtimes B & =A\boxtimes B+A'\boxtimes B, \\
        A\boxtimes (B+B') & =A\boxtimes B+A\boxtimes B',
    \end{array}
    \]
    and for any $k\in K$
    \[
    (kA)\boxtimes B=k(A\boxtimes B)=A\boxtimes(kB).
    \]
\end{lemma}

Because $A\boxtimes B\in M_{mn}(K)$, it is natural to see under what conditions $A\boxtimes B$ is an idempotent. 
To begin with, we need to introduce the concept of tensor product.
We will not dwell on the construction of {\bf tensor product}\index{tensor product} and the definition of {\bf bilinear map}\index{bilinear map} here.
For details on the tensor product, see \cite[Section~IV.5]{algebra}.
We just list the necessary statement.
Note that for any two $R$-modules $U$ and $V$, we write $U\otimes_R V$\index{$\otimes_R$: tensor product over $R$} to denote the tensor product of $U$ and $V$ over $R$.

\begin{theorem}
    [{\cite[Theorem~IV.5.6]{algebra}}]
    \label{thm: R module homomorphism of tensors induced by a bilinear map}
    If $U,V,W$ are modules over a commutative ring $R$ and $f:U\times V\to W$ is a bilinear map, then there is a unique $R$-module homomorphism $\overline{f}:U\otimes_R V\to W$ such that $\overline{f}\circ \iota= f$, where $\iota:U\times V\to U\otimes_R V$ is the canonical bilinear map. 
\[
    \begin{tikzcd}
    U\times V \arrow{r}{\iota} \arrow[swap]{d}{f} & U \otimes_R V \arrow{ld}{\overline{f}}\\
    W&  
    \end{tikzcd}
\]
    The module $U\otimes_R V$ is uniquely determined up to isomorphism by this property.
\end{theorem}

Consider the tensor product $M_m(K)\otimes_K M_n(K)$. That is, $R=K, U=M_m(K)$ and $V=M_n(K)$. 

\begin{proposition}
    Let $K$ be a field. Then there exists a $K$-module isomorphism $\overline{f}:M_m(K)\otimes_K M_n(K)\to M_{mn}(K)$ satisfying
    \[
    \begin{array}{cccc}
     \overline{f}:&M_m(K)\otimes_K M_n(K)&\to& M_{mn}(K) \\
     &A\otimes B&\mapsto&A\boxtimes B 
\end{array},
    \]
    where $A\in M_m(K)$ and $B\in M_n(K)$.
\end{proposition}
\begin{proof}
    Define $f:M_m(K)\times M_n(K)\to M_{mn}(K)$ by $f:(A,B)\mapsto A\boxtimes B$. Then by {Lemma~\ref{lem: Kronecker product is bilinear}}, we know that $f$ is a bilinear map.
    By {Theorem~\ref{thm: R module homomorphism of tensors induced by a bilinear map}}, we know that there is a unique $K$-module homomorphism $\overline{f}:M_m(K)\otimes_K M_n(K)\to M_{mn}(K)$ such that $f=\overline{f}\circ\iota$. Then $\overline{f}(A\otimes B)=\overline{f}(\iota(A,B))=(\overline{f}\circ \iota)(A,B)=f(A,B)=A\boxtimes B$.

    Let $E^{(m)}_{ij},E^{(n)}_{rs}$, and $E^{(mn)}_{uv}$ be the matrix units in $M_m(K),M_n(K)$, and $M_{mn}(K)$, respectively. Because $E^{(m)}_{ij}\boxtimes E^{(n)}_{rs}=E^{(mn)}_{(i-1)n+r,(j-1)n+s}$, we know that $\{E^{(m)}_{ij}\boxtimes E^{(n)}_{rs}\mid 1\le i,j\le m,1\le r,s\le n\}=\{E^{(mn)}_{uv}\mid 1\le u,v\le mn\}$ is a basis for $M_{mn}(K)$.
    Thus, $\overline{f}$ is a surjective map since $\overline{f}(E^{(m)}_{ij}\otimes E^{(n)}_{rs})=E^{(m)}_{ij}\boxtimes E^{(n)}_{rs}$ and that $\overline{f}$ is a $K$-homomorphism.
    Because ${\rm dim}_K(M_m(K)\otimes_K M_n(K))=mn={\rm dim}_K(M_{mn}(K))$ and that $\overline{f}$ is a $K$-module epimorphism, $\overline{f}$ must be $K$-module isomorphism.
\end{proof}

Note that $M_n(K)\otimes_K M_m(K)$ has a $K$-algebra structure, where the multiplication of $A\otimes B$ and $C\otimes D$ is given by $(A\otimes B)(C\otimes D)=AC\otimes BD$.
From the $K$-module isomorphism $\overline{f}$, it suggests that
\[
(A\boxtimes B)(C\boxtimes D)=AC\boxtimes BD
\]
for $A,C\in M_m(K)$ and $B,D\in M_n(K)$.
(Equivalently, $\overline{f}$ is a $K$-algebra homomorphism.)

\begin{proposition}
    Let $A,C\in M_m(K)$ and $B,D\in M_n(K)$. Then
    \[
    (A\boxtimes B)(C\boxtimes D)=AC\boxtimes BD.
    \]
    In particular, $\overline{f}$ is a $K$-algebra isomorphism.
\end{proposition}
\begin{proof}
    Write $E_{ij}$ to be the matrix units in $M_m(K)$.
    Because $E_{ij}E_{k\ell}=\delta_{jk}E_{i\ell}$, we know that
    \[
    (E_{ij}\boxtimes B)(E_{k\ell}\boxtimes D)=\delta_{jk}(E_{i\ell}\boxtimes BD)=(\delta_{jk}E_{i\ell})\boxtimes BD=E_{ij}E_{k\ell}\boxtimes BD.
    \]
    Let $A=\sum_{i,j}a_{ij}E_{ij}$ and $C=\sum_{k,\ell}c_{k\ell}E_{k\ell}$.
    Since the Kronecker product is $K$-bilinear, we have $\left(\sum_{i,j}a_{ij}E_{ij}\right)\boxtimes B=\sum_{i,j}a_{ij}(E_{ij}\boxtimes B)$ and $\left(\sum_{k,\ell}c_{k\ell}E_{k\ell}\right)\boxtimes D=\sum_{k,\ell}c_{k\ell}(E_{k\ell}\boxtimes D)$.
    Thus,
    \[
    \begin{array}{ll}
        (A\boxtimes B)(C\boxtimes D) 
         & =\left(\sum_{i,j}a_{ij}(E_{ij}\boxtimes B)\right)\left(\sum_{k,\ell}c_{k\ell}(E_{k\ell}\boxtimes D)\right) \\
         & =\sum_{i,j,k,\ell}a_{ij}c_{k\ell}(E_{ij}\boxtimes B)(E_{k\ell}\boxtimes D)\\
         & =\sum_{i,j,k,\ell}a_{ij}c_{k\ell}\delta_{jk}(E_{i\ell}\boxtimes BD)\\
         & =\sum_{i,\ell}\sum_{j}a_{ij}c_{j\ell}(E_{i\ell}\boxtimes BD)\\
         & =(\sum_{i,\ell}\sum_{j}a_{ij}c_{j\ell}E_{i\ell})\boxtimes BD\\
         &= AC\boxtimes BD.
    \end{array}
    \]
\end{proof}

\begin{proposition}
    Let $A\in M_m(K)$ and $B\in M_n(K)$ be idempotents. Then $A\boxtimes B$ is an idempotent in $M_{mn}(K)$.
\end{proposition}
\begin{proof}
    Because $A^2=A$ and $B^2=B$, we have $(A\boxtimes B)^2=A^2\boxtimes B^2=A\boxtimes B$. Thus, $A\boxtimes B$ is an idempotent.
\end{proof}

However, the converse may not hold in general. In other words, there exists $A$ and $B$ are not idempotents, but $A\boxtimes B$ is an idempotent.

\begin{example}
    Let $K={\mathbb Q}$.
    Consider $A=\left(\begin{array}{cc}
       1/2  & 0 \\
        0 & 0
    \end{array}\right)$ and $B=\left(\begin{array}{cc}
        2 & 0 \\
        0 & 0
    \end{array}\right)$.
    Then we have
    \[
    A\boxtimes B=\left(\begin{array}{cccc}
         1&0&0&0  \\
         0&0&0&0  \\
         0&0&0&0  \\
         0&0&0&0
    \end{array}\right)
    \]
    which is an idempotent.
    \hfill $\square$
\end{example}

Because $M_m(K)$ and $M_n(K)$ are vector space over $K$, we know that $A\otimes B=0$ if and only if $A=O$ or $B=O$.
In the Kronecker product version, we have $A\boxtimes B=O$ if and only if $A=O$ or $B=O$.

\begin{proposition}
    \label{prop: the necessary condition for A tensor B to be idempotent}
    Let $A\in M_m(K)$ and $B\in M_n(K)$. Suppose $A\neq O$ and $A\boxtimes B$ are idempotents, then $B$ must be an idempotent.
    Suppose $B\neq O$ and $A\boxtimes B$ are idempotents, then $A$ must be an idempotent.
\end{proposition}
\begin{proof}
    Because $A$ and $A\boxtimes B$ are idempotents,
    \[
    A\boxtimes B=(A\boxtimes B)^2=A^2\boxtimes B^2=A\boxtimes B^2.
    \]
    Thus,
    \[
    O=A\boxtimes B-A\boxtimes B^2=A\boxtimes (B-B^2).
    \]
    Since $A\neq O$, we can conclude that $B-B^2=O$ and so $B$ is an idempotent.
\end{proof}

\begin{theorem}
    Let $A\in M_m(K)$ and $B\in M_n(K)$. Then $A\boxtimes B$ is a nonzero idempotent if and only if $A$ and $B$ are both nonzero and $A^2=kA$ and $B^2=k^{-1}B$ for some $k\in K^{\times}$.
\end{theorem}
\begin{proof}
    ($\Leftarrow$) Because $A\neq O$ and $B\neq O$, we know that $A\boxtimes B$ is nonzero.
    Also,
    \[
    (A\boxtimes B)^2=A^2\boxtimes B^2=kA\boxtimes k^{-1}B=A\boxtimes B.
    \]
    Thus, $A\boxtimes B$ is a nonzero idempotent in $M_{mn}(K)$.

    ($\Rightarrow$) Write $C=A^2$. Because $C\boxtimes B^2=A^2\boxtimes B^2=(A\boxtimes B)^2=A\boxtimes B$, we have
    \[
    \left(\begin{array}{cccc}
     c_{11}B^2&c_{12}B^2&\cdots&c_{1m}B^2  \\
     c_{21}B^2&c_{22}B^2&\cdots&c_{2m}B^2 \\
     \vdots&\vdots&\ddots&\vdots \\
     c_{m1}B^2&c_{m2}B^2&\cdots&c_{mm}B^2
     \end{array}\right)=
     \left(\begin{array}{cccc}
     a_{11}B&a_{12}B&\cdots&a_{1m}B  \\
     a_{21}B&a_{22}B&\cdots&a_{2m}B \\
     \vdots&\vdots&\ddots&\vdots \\
     a_{m1}B&a_{m2}B&\cdots&a_{mm}B
     \end{array}\right)
    \]
    We claim that there exists $(i,j)$ such that $a_{ij}$ and $c_{ij}$ are both nonzero.
    Suppose not. That is, either $a_{ij}=0$ or $c_{ij}=0$ for all $1\le i,j\le n$.
    Because $A\boxtimes B$ is a nonzero idempotent, we have $A^2\boxtimes B^2\neq O$ and so $A^2\neq O$ and $B^2\neq O$.
    Because $A\neq O$ and $C=A^2\neq O$ and $a_{ij}B=c_{ij}B^2$ for all $(i,j)$, we know that either $B=O$ or $B^2=O$. Then there is a contradiction.
    Therefore, there exists $(i,j)$ such that $a_{ij}\neq 0$ and $c_{ij}\neq 0$.

    Let $k=c_{ij}/a_{ij}$. Because $a_{ij}B=c_{ij}B^2$, we have $B^2=k^{-1}B$.
    Also, we have
    \[
    A\boxtimes B=(A\boxtimes B)^2=A^2\boxtimes B^2=A^2\boxtimes k^{-1} B=k^{-1}A^2\boxtimes B.
    \]
    Thus, we know that
    \[
    O=A\boxtimes B-k^{-1}A^2\boxtimes B=(A-k^{-1}A^2)\boxtimes B.
    \]
    Since $B\neq O$, by the discussion before {Proposition~\ref{prop: the necessary condition for A tensor B to be idempotent}}, we know that $A-k^{-1}A^2=O$.
    Hence, $A^2=kA$.
\end{proof}

For a matrix $A$, we write the {\bf transpose of} $A$\index{transpose} to be $A^{\mathsf T}$.\index{$A^{\mathsf T}$: transpose of $A$} For the sake of the later discussion, we give the definition of anti-transpose.

\begin{definition}[{\cite[Definition~2.6]{antidiagonal}}]
    Let $A\in M_n(K)$. Define $^{\tau}A$\index{$^{\tau}A$: anti-transpose of $A$} to be the matrix taking transpose along the anti-diagonal. More precisely,
    \[
    {^{\tau}A}_{ij}=A_{n+1-j,n+1-i}.
    \]
    We say that $^{\tau}A$ is the {\bf anti-transpose of} $A$.\index{anti-transpose}
\end{definition}

From the remark after \cite[Definition~2.6]{antidiagonal}, we have the following lemma.

\begin{lemma}
    \label{lem: properties of exchange matrix}
    Let $J\in M_n(K)$ be the $n\times n$ exchange matrix
    \[
    J=\left(
    \begin{array}{ccccc}
         0&0&\cdots&0&1  \\
         0&0&\cdots&1&0  \\
         \vdots&\vdots&\reflectbox{$\ddots$}&\vdots&\vdots  \\
         0&1&\cdots&0&0  \\
         1&0&\cdots&0&0  
    \end{array}
    \right).
    \]
    i.e., the entries of $J$ satisfies $J_{ij}=\left\{\begin{array}{ll}
        1, &{\rm if}\ i+j=n+1 \\
        0, & {\rm otherwise}
    \end{array}\right.$.
    Then for any matrix $A\in M_n(K)$, we have
    \begin{enumerate}[(a)]
        \item $J^{-1}=J$
        \item $JAJ^{-1}=JAJ$ is the matrix rotate $\pi$ on $A$.
        \item $^{\tau}A=JA^{\mathsf T}J^{-1}=JA^{\mathsf T}J$.
        \item $(^{\tau}A)^{\mathsf T}={^{\tau}(A^{\mathsf T})}$
    \end{enumerate}
\end{lemma}
\begin{proof}
    The main idea is that given a matrix $A\in M_n(K)$ the left multiplication of $J$ on $A$ i.e., $JA$ is just reverse the order of rows of $A$, and the right multiplication of $J$ on $A$ i.e., $AJ$ is just reverse the order of columns of $A$.
    \begin{enumerate}[(a)]
        \item Let $B=\left(\begin{array}{cccc}
             {\bf b}_1&{\bf b}_2&\cdots&{\bf b}_n   
        \end{array}\right)$, then $BJ=\left(\begin{array}{cccc}
             {\bf b}_n&\cdots&{\bf b}_2 &{\bf b}_1   
        \end{array}\right)$.
        \\
        Thus, $J^2=\left(\begin{array}{cccc}
             {\bf e}_n&\cdots& {\bf e}_2 &{\bf e}_1   
        \end{array}\right)J=\left(\begin{array}{cccc}
             {\bf e}_1&{\bf e}_2&\cdots&{\bf e}_n   
        \end{array}\right)=I_n$ and so $J^{-1}=J$.
        \item Let $A=(a_{ij})$, then
        \[
        \begin{array}{ll}
        J\left(\begin{array}{cccc}
            a_{11} & a_{12} &\cdots &a_{1n} \\
            a_{21} & a_{22} &\cdots&a_{2n} \\
            \vdots&\vdots&\ddots&\vdots\\
            a_{n1} & a_{n2} &\cdots &a_{nn}
        \end{array}\right)J
             & =\left(\begin{array}{cccc}
            a_{n1} & a_{n2} &\cdots &a_{nn} \\
            \vdots & \vdots & \reflectbox{$\ddots$} &\vdots \\
            a_{21}&a_{22}&\cdots& a_{2n}\\
            a_{11} &a_{12} &\cdots &a_{1n}
        \end{array}\right)J\\
             & =\left(\begin{array}{cccc}
            a_{nn} & \cdots &a_{n2} &a_{n1} \\
            \vdots&\ddots&\vdots&\vdots\\
            a_{2n} & \cdots & a_{22} &a_{21} \\
            a_{1n} &\cdots & a_{12} &a_{11}
        \end{array}\right)
        \end{array}.
        \]
        \item Let $A=(a_{ij})$, then
        \[
        {^{\tau}A}=\left(\begin{array}{cccc}
             a_{nn}&\cdots&a_{2n}&a_{1n}  \\
             \vdots&\ddots&\vdots&\vdots  \\
             a_{n2}&\cdots&a_{22}&a_{12}  \\
             a_{n1}&\cdots&a_{21}&a_{11}
        \end{array}\right).
        \]
        On the other hand, by applying part (b), we know that $JA^{\mathsf T}J$ is
        \[
        J\left(\begin{array}{cccc}
            a_{11} & a_{21} &\cdots &a_{n1} \\
            a_{12} & a_{22} &\cdots&a_{n2} \\
            \vdots&\vdots&\ddots&\vdots\\
            a_{1n} & a_{2n} &\cdots &a_{nn}
        \end{array}\right)J
        =
        \left(\begin{array}{cccc}
             a_{nn}&\cdots&a_{2n}&a_{1n}  \\
             \vdots&\ddots&\vdots&\vdots  \\
             a_{n2}&\cdots&a_{22}&a_{12}  \\
             a_{n1}&\cdots&a_{21}&a_{11}
        \end{array}\right).
        \]
        Therefore, ${^{\tau}A}=JA^{\mathsf T}J$.
        \item $({^{\tau}A})^{\mathsf T}=(JA^{\mathsf T}J)^{\mathsf T}=J^{\mathsf T}(A^{\mathsf T})^{\mathsf T}J^{\mathsf T}=J(A^{\mathsf T})^{\mathsf T}J={^{\tau}(A^{\mathsf T})}$.
        The first and the last equations come from part (c).
    \end{enumerate}
\end{proof}

\begin{corollary}
    \label{cor: transposition reverse the multiplicative order}
    Let $A,B\in M_n(K)$.
    The anti-transpose has an analogue to the transpose that $(AB)^{\mathsf T}=B^{\mathsf T}A^{\mathsf T}$ as follows:
    \[
    {^{\tau}(AB)}=(^{\tau}B)(^{\tau}A).
    \]
\end{corollary}
\begin{proof}
    By {Lemma~\ref{lem: properties of exchange matrix}} (c), we obtain ${^{\tau}(AB)}=J(AB)^{\mathsf T}J^{-1}=JB^{\mathsf T}A^{\mathsf T}J^{-1}=JB^{\mathsf T}J^{-1}JA^{\mathsf T}J^{-1}=(^{\tau}B)(^{\tau}A)$.
\end{proof}

\begin{corollary}
    \label{cor: transpose preserve idempotent}
    Let $E\in M_n(K)$, then the following are equivalent. 
    \begin{enumerate}[(i)]
        \item $E$ is an idempotent
        \item $E^{\mathsf T}$ is an idempotent
        \item ${^{\tau}E}$ is an idempotent
        \item $(^{\tau}E)^{\mathsf T}={^{\tau}(E^{\mathsf T})}$ is an idempotent.
    \end{enumerate}
\end{corollary}
\begin{proof}
    It is sufficient to show the following argument: if $E$ is an idempotent, then $E^{\mathsf T}$ and $^{\tau}E$ are idempotents.
    Because $E$ is an idempotent, $E^2=E$. By {Corollary~\ref{cor: transposition reverse the multiplicative order}} we have $(E^2)^{\mathsf T}=(E^{\mathsf T})^2$ and ${^{\tau}(E^2)}=({^{\tau}E})^2$.
    Thus, $(E^{\mathsf T})^2=(E^2)^{\mathsf T}=E^{\mathsf T}$ and $({^{\tau}E})^2={^{\tau}(E^2)}={^{\tau}E}$.
    That is, we have shown $(i)\Rightarrow (ii)$ and $(i)\Rightarrow (iii)$.

    $(ii)\Rightarrow (i)$ and $(iii)\Rightarrow (iv)$: Just apply transpose of the both side again.

    $(iv)\Rightarrow (ii)$: Just apply anti-transpose of the both side again.
\end{proof}

\section{As an Affine Algebraic Variety}
\label{sec: As an Affine Algebraic Variety}

Let $\{F_i\}_{\in I}$ be a collection of polynomials in $K[x_1,...,x_n]$ for a field $K$. The set of common zeros of all $F_i$ is called an {\bf affine algebraic variety}, denoted as ${\mathbb V}(\{F_i\}_{i\in I})$. Specifically, 
\[
{\mathbb V}(\{F_i\}_{i\in I})=\{(a_1,...,a_n)\in K^n\mid F_i((a_1,...,a_n))=0,\ {\textrm {for all}}\ i\in I\}.
\]
If $X=(x_{ij})$ is an $n \times n$ matrix, then to solve the matrix equation $X^2=X$ is equivalent to finding common zero set of the polynomials $F_{ij}=(\sum_{k}x_{ik}x_{kj})-x_{ij}$ for $1\le i,j\le n$. Thus, one can identify ${\mathscr I}(M_n(K))$ as the variety ${\mathbb V}(\{F_{ij}\mid 1\le i,j\le n\})$ by viewing $M_n(K)$ as $K^{n^2}$. As a convention, we write $K^n$ as ${\mathbb A}^n$, which is also called an {\bf affine} $n$-{\bf space}.

In this section, $K$ is assumed to be an algebraically closed field. A range of the dimension of $M_n(K)$ as an affine algebraic variety will be discussed.

Let $X=(x_{ij})\in M_n(K)$. We can consider the affine algebraic variety defined by the polynomial $\operatorname{tr}(X)-r=(\sum_{j=1}^{n}x_{jj})-r$ i.e., 
\[
{\mathcal V}_r:={\mathbb V}(\operatorname{tr}(X)-r)=\{X\in M_n(K)\mid \operatorname{tr}(X)=r\}.
\]
We first assume that $\operatorname{char}(K)=0$.
Note that if $E$ is an idempotent matrix, then we have $\operatorname{tr}(E)=\operatorname{rank}(E)$.
Thus, for $r\in\{0,1,...,n\}$,
\[
{\mathscr I}(M_n(K))\cap {\mathcal V}_r=\{E\in M_n(K)\mid \operatorname{rank}(E)=r\}
\]
is an affine algebraic variety.
Moreover, ${\mathscr I}(M_n(K))$ is the union of proper subvarieties ${\mathscr I}(M_n(K))\cap {\mathcal V}_r$ for $0 \leq r \leq n$. 
On the other hand, if $\operatorname{char}(K)=p\neq 0$, then $\operatorname{tr}(E)\equiv\operatorname{rank}(E)\ (\operatorname{mod}p)$ and so
\[
{\mathscr I}(M_n(K))\cap {\mathcal V}_r=\{E\in M_n(K)\mid \operatorname{rank}(E)\equiv r\ (\operatorname{mod}p)\}.
\]
Moreover, ${\mathscr I}(M_n(K))$ is the union of proper subvarieties ${\mathscr I}(M_n(K))\cap {\mathcal V}_r$ for $r \in \mathbb{N}$. 

We have the following immediate observation, for which the definition of an irreducible variety can be found in {\cite[page 12]{AnInvitationToAlgebraicGeometry}}.

\begin{proposition}
    ${\mathscr I}(M_n(K))$ is not an irreducible variety.
\end{proposition}

\begin{definition}[{\cite[page 12]{AnInvitationToAlgebraicGeometry}}]
    Let $V\subseteq {\mathbb A}^n$ be an affine algebraic variety.
    The {\bf {dimension of}} $V$ is defined to be the length $d$ of the longest possible chain of proper irreducible subvarieties of $V$,
    \[
    V\supseteq V_d\supsetneq V_{d-1}\supsetneq\cdots\supsetneq V_1\supsetneq V_0
    \]
    where $V_0,V_1,...,V_d$ are irreducible subvarieties of $V$.
\end{definition}

The following question then arises naturally: What is the dimension of the affine algebraic variety $\mathscr{I}(M_n(K))$?
At present, we are unable to provide an exact value; however, a range can be specified.

The following lemma can be found in \cite[Theorem~1.19]{BasicAlgebraicGeometry1}. However, the definition of dimension in Shafarevich's book 
is different from the one above, although they are equivalent. We provide a proof here for convenience.

\begin{lemma}
    Let $X,Y\subseteq {\mathbb A}^n$ be affine algebraic varieties. If $X\subseteq Y$, then $\dim(X)\le \dim(Y)$. If $Y$ is irreducible and $X\subseteq Y$ with $\dim(X)=\dim(Y)$, then $X=Y$.
\end{lemma}
\begin{proof}

    Because $X\subseteq Y$, we can find irreducible subvarieties $V_0,V_1,...,V_{\dim(X)}$ of $Y$ such that
    \[
    Y\supseteq X\supseteq V_{\dim(X)}\supsetneq V_{\dim(X)-1}\supsetneq\cdots\supsetneq V_1\supsetneq V_0.
    \]
    Thus, $\dim(Y)\ge \dim(X)$.
    If $Y$ is irreducible and $X \subsetneq Y$, then $V_{\dim(X)} \subsetneq Y$ and $\dim(Y)\ge \dim(X)+1$ by definition. This proves the second statement.
\end{proof}

Note that $\dim({\mathbb A}^n)=n$ (\cite[Example~1.30]{BasicAlgebraicGeometry1}).
If $V$ and $W$ are isomorphic as affine algebraic varieties, then $\dim(V)=\dim(W)$.
For the definition of a morphism and an isomorphism between two affine algebraic varieties, see {\cite[page 10]{AnInvitationToAlgebraicGeometry}}.
The following result seems to be known to experts, but we could not find explicitly in the literature.

\begin{lemma}
    \label{lem: variety dimesion of a vector space}
    Let $V\subseteq {\mathbb A}^n$ be a subspace over $K$. Then the dimension of $V$ as a vector space coincides the dimension of $V$ as an affine algebraic variety.
\end{lemma}
\begin{proof}
    Let $d$ be the dimension of $V$ as a vector space. We need to find morphisms $\phi:V\to {\mathbb A}^d$ and $\psi:{\mathbb A}^d\to V$ such that $\psi\circ\phi=\operatorname{id}_V$ and $\phi\circ\psi=\operatorname{id}_{{\mathbb A}^d}$.
    Let $\{{\bf v}_1,...,{\bf v}_d\}$ be a basis for $V$ where ${\bf v}_i=(v_{1i},v_{2i},...,v_{ni})^{\mathsf T}$.
    Define $\psi:{\mathbb A}^d\to V$ by
    \[
    \psi(\left(\begin{array}{c}
          x_1 \\
          x_2 \\
          \vdots \\
          x_d 
    \end{array}\right))=\left(\begin{array}{c}
          \sum_{j=1}^{d}v_{1j}x_j \\
          \sum_{j=1}^{d}v_{2j}x_j \\
          \vdots \\
          \sum_{j=1}^{d}v_{nj}x_j 
    \end{array}\right)=x_1{\bf v}_1+\cdots+x_d{\bf v}_d.
    \]
    It is clear that $\psi$ is a morphism.
    Because $\{{\bf v}_1,...,{\bf v}_d\}$ is a basis for $V$, for any ${\bf y}=(y_1,...,y_n)^{\mathsf T}\in V$, there exist unique $c_1,...,c_d\in K$ such that ${\bf y}=c_1{\bf v}_1+\cdots+c_d{\bf v}_d$.
    Thus, the map $\phi:V\to {\mathbb A}^d$ defined by $\phi({\bf y})=(c_1,...,c_d)^{\mathsf T}$ is a well-defined function.
    We need to show that $c_i$ can be written as a polynomial in $y_1,...,y_n$.
    Consider the matrix
    \[
    B=\left(\begin{array}{cccc}
         {\bf v}_1&{\bf v}_2&\cdots&{\bf v}_d  
    \end{array}\right)=
    \left(\begin{array}{cccc}
         v_{11}&v_{12}&\cdots&v_{1d}  \\
         v_{21}&v_{22}&\cdots&v_{2d}  \\
         \vdots&\vdots&\ddots&\vdots  \\
         v_{n1}&v_{n2}&\cdots&v_{nd}
    \end{array}\right).
    \]
    Since $\operatorname{rank}(B)=d$, we can choose $d$ linearly independent rows of $B$ to form a $d\times d$ invertible matrix
    \[
    C=\left(\begin{array}{cccc}
         v_{i_11}&v_{i_12}&\cdots&v_{i_1d}  \\
         v_{i_21}&v_{i_22}&\cdots&v_{i_2d}  \\
         \vdots&\vdots&\ddots&\vdots  \\
         v_{i_d1}&v_{i_d2}&\cdots&v_{i_dd}
    \end{array}\right).
    \]
    Then we have
    \[
    \left(\begin{array}{c}
         y_{i_1}  \\
         y_{i_2}  \\
         \vdots  \\
         y_{i_d}  
    \end{array}\right)=
    \left(\begin{array}{cccc}
         v_{i_11}&v_{i_12}&\cdots&v_{i_1d}  \\
         v_{i_21}&v_{i_22}&\cdots&v_{i_2d}  \\
         \vdots&\vdots&\ddots&\vdots  \\
         v_{i_d1}&v_{i_d2}&\cdots&v_{i_dd}
    \end{array}\right)
    \left(\begin{array}{c}
         c_{1}  \\
         c_{2}  \\
         \vdots  \\
         c_{d}  
    \end{array}\right)
    \]
    and so
    \[
    \left(\begin{array}{c}
         c_{1}  \\
         c_{2}  \\
         \vdots  \\
         c_{d}  
    \end{array}\right)=C^{-1}
    \left(\begin{array}{c}
         y_{i_1}  \\
         y_{i_2}  \\
         \vdots  \\
         y_{i_d}  
    \end{array}\right).
    \]
    Therefore, $c_i$ can be written as a polynomial in $y_1,...,y_n$ and so $\phi$ is a morphism.
    It is clear that $\psi\circ\phi=\operatorname{id}_V$ and $\phi\circ\psi=\operatorname{id}_{{\mathbb A}^d}$.
    Hence, $V\simeq {\mathbb A}^d$ and so $\dim(V)=d$.
\end{proof}

\begin{lemma}
    For $r\in\{0,1,...,n\}$, ${\mathcal V}_r$ is irreducible. Futhermore, $\dim({\mathcal V}_r)=n^2-1$.
\end{lemma}
\begin{proof}
    Note that ${\mathcal V}_r={\mathbb V}(I)$ where $I=\langle\operatorname{tr}(X)-r\rangle$. Since $\operatorname{tr}(X)-r$ is a polynomial of degree $1$ in $K[x_{11},...,x_{nn}]$, $\operatorname{tr}(X)-r$ is an irreducible element (and so a prime element as $K[x_{11},...,x_{nn}]$ is a UFD). Thus, $I$ is a prime ideal and so a radical ideal. By Hilbert's Nullstellensatz theorem({\cite[Proposition~15.32]{AbstractAlgebra}}) and {\cite[Proposition~15.17]{AbstractAlgebra}}, we know that ${\mathcal V}_r={\mathbb V}(I)$ must be an irreducible variety. Moreover, it is isomorphic to ${\mathcal V}_0$, since it is a translation of ${\mathcal V}_0$.
    By {Lemma~\ref{lem: variety dimesion of a vector space}}, we know that $\dim({\mathcal V}_0)=n^2-1$, as desired. 
\end{proof}

Since ${\mathscr I}(M_n(K))\cap {\mathcal V}_r\subsetneq {\mathcal V}_r$ for all $r$, we have
\[
\dim({\mathscr I}(M_n(K))\cap {\mathcal V}_r)\lneq \dim({\mathcal V}_r)=n^2-1,\quad n\ge 2.
\]
The following result seems to be known to experts, but we could not find explicitly in the literature.

\begin{lemma}
    \label{lem: dimension of union of varieties}
    Let $X,Y$ be two affine algebraic varieties. Then $\dim(X\cup Y)=\max\{\dim(X),\dim(Y)\}$.
\end{lemma}
\begin{proof}
    Suppose $\dim(X\cup Y)=d$.
    Let $X\cup Y\supseteq V_d\supsetneq V_{d-1}\supsetneq\cdots\supsetneq V_1\supsetneq V_0$ be a longest possible chain of irreducible subvarieties of $X\cup Y$.
    Then either $V_d\subseteq X$ or $V_d\subseteq Y$. Suppose this is not the case. Then $X\cap V_d\neq\varnothing$ and $Y\cap V_d\neq \varnothing$.
    Thus, $V_d=(X\cap V_d)\cup (Y\cap V_d)$ and so $V_d$ is not irreducible. Then there is a contradiction.
    Therefore, we can conclude that $d=\max\{\dim(X),\dim(Y)\}$.
\end{proof}

Combining the above two lemmas, we have the following conclusion. 

\begin{proposition}
    \label{prop: variety dimension of I(M_n(K))}
For $n \geq 2$, $n-1\le\dim({\mathscr I}(M_n(K)))\le n^2-2$.
\end{proposition}
\begin{proof}
    The upper bound was obtained in previous discussion. We focus on the lower bound.
    Let $s_2,...,s_n\in K$ be given. By {Corollary~\ref{cor: idempotent in Mn(UFD) of rank 1}}, consider the set $W\subseteq {\mathscr I}(M_n(K))\cap {\mathcal V}_1$ which is given by
    \[
    W=\left\{\left(\begin{array}{cccc}
         x_{11}&x_{12}&\cdots&x_{1n}  \\
         s_2x_{11}&s_2x_{12}&\cdots&s_2x_{1n}  \\
         s_3x_{11}&s_3x_{12}&\cdots&s_3x_{1n}  \\
         \vdots&\vdots&\ddots&\vdots\\
         s_nx_{11}&s_nx_{12}&\cdots&s_nx_{1n}
    \end{array}\right)\ \middle|\  x_{11}+\sum_{j=2}^{n}s_jx_{1j}=1 \right\}.
    \]
    We claim that $\dim(W) = n-1$. In fact, $W$ is a translation of a subspace $W_0$ translated by one unit (That is, $W=W_0+E_{11}$ and $E_{11}$ is a matrix unit). Moreover, $W_0$ can be also viewed as the variety 
    \[
    W_0={\mathbb V}\left(\{F_{ij}\mid2\le i\le n,1\le j\le n\}\cup\left\{x_{11}+\sum_{j=2}^{n}s_jx_{1j}\right\}\right)
    \]
    and $F_{ij}=F_{ij}=x_{ij}-s_ix_{1j}$. Thus, $W$ is an affine algebraic variety and its dimension is the same as $\dim(W_0)$. In particular, $\dim(W_0)$ is the same as the dimension as a vector space, by {Lemma~\ref{lem: variety dimesion of a vector space}}.
    This can be obtained by the row echelon form with the variables that are ordered by $(x_{nn},...,x_{ij},x_{i(j-1)},...,x_{i1},x_{(i-1)n},...,x_{11})$.
    The coefficient matrix is
    \[
    \left(\begin{array}{ccccc|c}
         I_n&O&\cdots&O&O&-s_nI_n  \\
         O&I_n&\cdots&O&O&-s_{n-1}I_n  \\
         \vdots&\vdots&\ddots&\vdots&\vdots&\vdots  \\
         O&O&\cdots&I_n&O&-s_3I_n\\
         O&O&\cdots&O&I_n&-s_2I_n  \\
         \hline 
         &&A&&&B
    \end{array}\right)
    \]
    where $A=(a_{ij})$ is a $1\times n(n-1)$ matrix and $B=(b_{ij})$ is an $1\times n$ matrix defined by
    \[
    a_{ij}=\left\{\begin{array}{cl}
        1 & {\textrm {if}}\ j=1+k(n+1)\ {\textrm {for}}\ 0\le k\le n-2 , \\
        0 & {\textrm{otherwise},} 
    \end{array}\right.
    \]
    and
    \[
    b_{ij}=\left\{\begin{array}{cl}
         1&{\textrm{if}}\ j=n,  \\
         0&{\textrm{otherwise}}, 
    \end{array}\right.
    \]
    respectively. Then a row echelon form of the above matrix is
    \[
    \left(\begin{array}{ccccc|c}
         I_n&O&\cdots&O&O&-s_nI_n  \\
         O&I_n&\cdots&O&O&-s_{n-1}I_n  \\
         \vdots&\vdots&\ddots&\vdots&\vdots&\vdots  \\
         O&O&\cdots&I_n&O&-s_3I_n\\
         O&O&\cdots&O&I_n&-s_2I_n  \\
         \hline 
         O&O&\cdots&O&O&S
    \end{array}\right)
    \]
    where 
    $S = (s_n \ s_{n-1} \cdots \ s_2 \ 1)$. Hence,
    $\dim(W_0) = n^2-(n(n-1)+1)=n-1$.
\end{proof}


It is clear that ${\mathscr I}(M_1(K))$ has dimension $0$. 
The preceding lemma implies that the dimension of ${\mathscr I}(M_2(K))$ is either $1$ or $2$. We conclude this section by determining the exact value of $\dim({\mathscr I}(M_2(K)))$.

\begin{example}
    We compute the dimension of ${\mathscr I}(M_2(K))$. 
    The following procedure and the terminology used are taken from \cite{IdealsVarietiesandAlgorithms}, in particular Theorem~9.3.8; the equivalence between the book's notion of dimension and ours is established by Definition~9.3.7 and Theorem~9.5.6.
    
    Let $x_{11}=a,x_{12}=b,x_{21}=c,x_{22}=d$, and
    \[
    I=\langle a^2+bc-a,ab+bd-b,ac+cd-c,bc+d^2-d\rangle\lhd K[a,b,c,d].
    \]
    Then ${\mathscr I}(M_2(K))={\mathbb V}(I)$.
    Define a grlex order by $a>b>c>d$ and find 
    a Gr\"obner basis $G$ for $I$ where 
    \[
    G=\{\underbrace{a^2+bc-a}_{f_1},\underbrace{ab+db-b}_{f_2},\underbrace{ac+cd-c}_{f_3},\underbrace{bc+d^2-d}_{f_4},\underbrace{ad^2+d^3-ad-2d^2+d}_{f_5=-cf_2+(a+d-1)f_4}\}.
    \]
    That is, $\langle G\rangle=I$ and 
    $\langle \operatorname{LT}(I)\rangle=\langle\operatorname{LT}(f_i) \mid 1 \leq i \leq 5 \rangle=\langle a^2,ab,ac,bc,ad^2\rangle$ where $\operatorname{LT}(f)$ is the leading term of $f$ under the given ordering and $\operatorname{LT}(I)=\{\operatorname{LT}(f)\mid f\in I\}$.
    By \cite[Theorem~9.3.8]{IdealsVarietiesandAlgorithms}, we know that $\dim({\mathbb V}(I))=\dim({\mathbb V}(\langle\operatorname{LT}(I)\rangle))=\dim({\mathbb V}(a^2,ab,ac,bc,ad^2))$.
    Note that ${\mathbb V}(a^2,ab,ac,bc,ad^2)={\mathbb V}(a^2)\cap {\mathbb V}(ab)\cap {\mathbb V}(ac)\cap {\mathbb V}(bc)\cap {\mathbb V(ad^2)}={\mathbb V(a)}\cap {\mathbb V}(bc)={\mathbb V}(a,b)\cup {\mathbb V}(b,c)$.
    It is clear that $\dim({\mathbb V}(a,b))=\dim({\mathbb V}(b,c))$, so we only need to compute the dimension of ${\mathbb V}(a,b)$ by {Lemma~\ref{lem: dimension of union of varieties}}.
    Because 
    \[{\mathbb V}(a,b)=\left\{\left(\begin{array}{cc}
        a & b \\
        c & d
    \end{array}\right)\ \middle|\ a=0,b=0,c,d\in K\right\}
    \]
    is a vector subspace of $M_2(K)$ of dimension $2$, by {Lemma~\ref{lem: variety dimesion of a vector space}} we know that $\dim({\mathbb V}(a,b))=2$.
    Hence, we can conclude that $\dim({\mathscr I}(M_2(K)))=\dim({\mathbb V}(I))=\dim({\mathbb V}(a,b))=2$.
    \hfill $\square$
\end{example}

\section{Concluding Remark}
It may be observed that {Section~\ref{sec: Constructing Idempotent Matrices}} is not directly relevant to the partial order relation.
The reason is that we don't know whether an idempotent matrix over an integral domain is diagonalizable or not.
We restrict our attention to PID because, at least in this case, we can utilize the Smith normal form. In other words, it is more likely that an idempotent matrix is diagonalizable.
We formulate the following problem (compare to {Proposition~\ref{diagonalization of idempotent matrix}}).

\begin{problem}
    Let $R$ be a PID and let $E$ be an idempotent matrix in $M_n(R)$. Does there exist an invertible matrix $A\in M_n(R)^{\times}$ such that
    \[
    E=A\left(\begin{array}{c|c}
         I_{{\rm rank}(E)}&O  \\
         \hline
         O&O 
    \end{array}\right)A^{-1}\ ?
    \]
\end{problem}

If the answer is true, then all properties in {Section~\ref{sec: The Poset Structure of Idempotents}} have PID version.
We can subsequently use this to characterize the partial order on ${\mathscr I}(M_n(R))$.

Now, let's continue introducing the next problem.
Let $R$ be a PID.
Recall that {Theorem~\ref{thm: idempotent in Mn(PID) in alternative form}} tells us that $E$ is an idempotent in $M_n(R)$ with ${\rm rank}(E)=\ell$ if and only if
\[
E=\left(\begin{array}{c|c}
     CA&CB  \\
     \hline
     DA& DB
\end{array}\right)
\]
for some $A,C\in M_{\ell}(R)$ and matrices $B,D$ such that $AC+BD=I_{\ell}$.
We wish to generalize to the situation that $R$ is a UFD.

\begin{problem}
    Let $R$ be a UFD. Let $E$ be an idempotent in $M_n(R)$ with ${\rm rank}(E)=\ell$. Does
    \[
    E=\left(\begin{array}{c|c}
         CA&CB  \\
         \hline
         DA& DB
    \end{array}\right)
    \]
    for some $A,C\in M_{\ell}(R)$ and matrices $B,D$ such that $AC+BD=I_{\ell}$?
\end{problem}

Finally, as an affine algebraic variety ${\mathscr I}(M_n(K))\subseteq {\mathbb A}^{n^2}$ where $K$ is an algebraically closed field, the range of its dimension is given by $n-1\le \dim({\mathscr I}(M_n(K)))\le n^2-2$ in {Proposition~\ref{prop: variety dimension of I(M_n(K))}}.
A computer program produces the dimensions $0, 2, 4, 8, 12, 18$ for $n = 1, 2, 3, 4, 5, 6$, respectively. 

\begin{problem}
    What is the dimension of the affine algebraic variety ${\mathscr I}(M_n(K))$?
\end{problem}

\section*{Acknowledgment}
This paper is a part of the second author's master's thesis written at the Department of Mathematics, National Taiwan Normal University. The research of the second author is financially supported by the Institute of Mathematics, Academia Sinica.
The second author would like to thank Professor Liang-Chung Hsia and Professor Chia-Hsin Liu for helpful discussions and continuous encouragement throughout this research.
The first author is partially supported by NSTC grant 113-2115-M-003-010-MY3. 
The third author is partially supported by NSTC grant 114-2115-M-017-001-MY2. 



\bibliographystyle{alphaurl}

\begin{thebibliography}{SKKT00}

\bibitem[CLD25]{IdealsVarietiesandAlgorithms}
David A. Cox, John Little, and Donal O'Shea.
\newblock {\em Ideals, Varieties, and Algorithms (5th ed.)}.
\newblock Springer, 2025.
\newblock \href {https://doi.org/10.1007/978-3-031-91841-4} {\path{doi:10.1007/978-3-031-91841-4}}.

\bibitem[Cri18]{numofidempotent}
Geoffrey Critzer.
\newblock {Combinatorics of Vector Spaces over Finite Fields}.
\newblock {\em Emporia State Institutional Repository Collection}, 2018.
\newblock URL: \url{https://esirc.emporia.edu/handle/123456789/3595}.

\bibitem[Cri22]{CountingMatrixOverFiniteFields}
Geoffrey Critzer.
\newblock {Counting Matrices Over Finite Fields}.
\newblock {\em Department of Mathematics, University of Kansas}, 2022.
\newblock URL: \url{https://hdl.handle.net/1808/33724}.

\bibitem[DF04]{AbstractAlgebra}
David S. Dummit and Richard M. Foote.
\newblock {\em Abstract Algebra (3rd ed.)}.
\newblock Wiley, 2004.
\newblock {\path{isbn:978-0471433347}}.

\bibitem[GS07]{antidiagonal}
Vasily Golyshev and Jan Stienstra.
\newblock {Fuchsian equations of type DN}.
\newblock {\em Communications in Number Theory and Physics}, 2007.
\newblock \href {https://doi.org/10.4310/CNTP.2007.v1.n2.a3} {\path{doi:10.4310/CNTP.2007.v1.n2.a3}}.

\bibitem[Hun80]{algebra}
Thomas~William Hungerford.
\newblock {\em Algebra}.
\newblock Springer, 1980.
\newblock \href {https://doi.org/10.1007/978-1-4612-6101-8} {\path{doi:10.1007/978-1-4612-6101-8}}.

\bibitem[Isa94]{mialgebra}
Irving~Martin Isaacs.
\newblock {\em Algebra: A Graduate Course}.
\newblock American Mathematical Society, 1994.

\bibitem[Jac53]{LecturesinAbstractAlgebraLinearAlgebra}
Nathan Jacobson.
\newblock {\em Lectures in Abstract Algebra II. Linear Algebra}.
\newblock Springer, 1953.
\newblock \href {https://doi.org/10.1007/978-1-4684-7053-6} {\path{doi:10.1007/978-1-4684-7053-6}}.

\bibitem[Lee16]{AShortProofofWedderburnArtinTheoremTK}
Tsiu-Kwen Lee.
\newblock {A short proof of the Wedderburn-Artin theorem}.
\newblock {\em Communications in Algebra}, 2016.
\newblock \href {https://doi.org/10.1080/00927872.2016.1233242} {\path{doi:10.1080/00927872.2016.1233242}}.

\bibitem[MS02]{groupring}
César~Polcino Milies and Sudarshan~K. Sehgal.
\newblock {\em An Introduction to Group Rings}.
\newblock Springer, 2002.

\bibitem[Nic93]{ashortproofofwedderburnartintheorem}
William~Keith Nicholson.
\newblock {A short proof of the Wedderburn-Artin theorem}.
\newblock {\em New Zealand J. Math}, 1993.

\bibitem[Sha13]{BasicAlgebraicGeometry1}
Igor~R. Shafarevich.
\newblock {\em Basic Algebraic Geometry 1}.
\newblock Springer, 2013.
\newblock \href {https://doi.org/10.1007/978-3-642-37956-7} {\path{doi:10.1007/978-3-642-37956-7}}.

\bibitem[SKKT00]{AnInvitationToAlgebraicGeometry}
Karen~E. Smith, Lauri Kahanpää, Pekka Kekäläinen, and William Traves.
\newblock {\em An Invitation to Algebraic Geometry}.
\newblock Springer, 2000.
\newblock \href {https://doi.org/10.1007/978-1-4757-4497-2} {\path{doi:10.1007/978-1-4757-4497-2}}.

\end{thebibliography}

\Addresses

\end{document}